\documentclass[12pt]{article}
\usepackage[margin=1in]{geometry}
\usepackage{amsthm, amsmath,amsfonts,amssymb,euscript,hyperref,graphics,color,slashed,mathrsfs}
\usepackage{graphicx}
\usepackage{comment}
\usepackage{import}
\usepackage{tikz}
\usepackage{latexsym}
\usepackage{mathtools}
\usepackage{appendix}
\usepackage{stmaryrd}
\usetikzlibrary{arrows}
\usepackage{pgfplots}
\usepackage{extarrows}

\allowdisplaybreaks

\makeatletter
\newcommand{\subjclass}[2][1991]{%
  \let\@oldtitle\@title%
  \gdef\@title{\@oldtitle\footnotetext{#1 \emph{Mathematics subject classification.} #2}}%
}
\newcommand{\keywords}[1]{%
  \let\@@oldtitle\@title%
  \gdef\@title{\@@oldtitle\footnotetext{\emph{Key words and phrases.} #1}}%
}
\makeatother

\newtheorem{theorem}{Theorem}[section]
\newtheorem*{theorem*}{Theorem}
\newtheorem{lemma}[theorem]{Lemma}
\newtheorem{proposition}[theorem]{Proposition}

\newtheorem{claim}[theorem]{Claim}
\theoremstyle{remark}
\newtheorem{definition}[theorem]{\bf Definition}
\newtheorem{remark}[theorem]{\bf Remark}

\numberwithin{equation}{section}
\numberwithin{figure}{section}

\begin{document}
\title {Small improvements to the Ball--Rivoal theorem and its $p$-adic variant}

\author{Li Lai}
\date{}
\subjclass[2020]{11J72 (primary), 11M06, 33C20 (secondary).}
\keywords{Linear independence, zeta values, $p$-adic zeta values, hypergeometric series.}

\maketitle

\begin{abstract}
We prove that the dimension of the $\mathbb{Q}$-linear span of the Riemann zeta values $\zeta(3),\zeta(5),\ldots,\zeta(s-1)$ is at least $(1.119 \cdot \log s)/(1+\log 2)$ for any sufficiently large even integer $s \geqslant s_0$. 
This result slightly refines the theorem of Rivoal (2000) or Ball--Rivoal (2001). 
The proof incorporates the arithmetic factor $\Phi_n$ introduced by Zudilin (2001) into a slight generalization of the Ball--Rivoal hypergeometric construction. 
The asymptotic impact of the arithmetic factor $\Phi_n$ was not anticipated prior to this work.
Although this result is subsumed by a much stronger recent development by Fischler (2026), our approach yields a new result in the $p$-adic setting, slightly refining a theorem of Sprang (2020).
We prove that the dimension of the $\mathbb{Q}$-linear span of the $p$-adic zeta values $\zeta_p(3),\zeta_p(5),\ldots,\zeta_p(s-1)$ is at least $(1.119 \cdot \log s)/(1+\log 2)$ for any prime $p$ and any sufficiently large even integer $s \geqslant s_0(p)$. 
\end{abstract}

\section{Introduction}

The arithmetic nature of special values of the Riemann zeta function $\zeta(\cdot)$ has attracted much attention. 
The values of $\zeta(\cdot)$ at positive even integers are well understood thanks to  Euler's formula, which shows that $\zeta(2k)$ is a non-zero rational multiple of $\pi^{2k}$ for any positive integer $k$. 
In contrast, the values of $\zeta(\cdot)$ at positive odd integers greater than one (referred to as \emph{odd zeta values} for brevity) are more mysterious. 
It is conjectured that $\pi$ and all the odd zeta values are algebraically independent over $\mathbb{Q}$. 
Toward this conjecture, Rivoal in 2000 (or Ball--Rivoal in 2001) made the following breakthrough:

\begin{theorem}[Rivoal \cite{Riv2000}, 2000; Ball--Rivoal \cite{BR2001}, 2001]
	We have
	\[ 
	\dim_\mathbb{Q}\operatorname{Span}_{\mathbb{Q}}\left( 1,\zeta(3),\zeta(5),\ldots,\zeta(s) \right) \geqslant \frac{1-o(1)}{1+\log 2} \cdot \log s 
	\]
	as the odd integer $s \to \infty$.  
\end{theorem}

In 2020, Sprang established a $p$-adic version of the Ball--Rivoal theorem.
He proved the following theorem: 

\begin{theorem}[Sprang \cite{Spr2020}, 2020]
	Let $p$ be any prime number. 
	We have
	\[
	\dim_\mathbb{Q}\operatorname{Span}_{\mathbb{Q}}\left( 1,\zeta_p(3),\zeta_p(5),\ldots,\zeta_p(s) \right) \geqslant \frac{1-o(1)}{2(1+\log 2)} \cdot \log s 
	\]
	as the odd integer $s \to \infty$.  
\end{theorem}

Here $\zeta_p(s)$ (for an odd integer $s>1$) is the $p$-adic zeta value. 
We will recall basic facts about $p$-adic zeta values in \S \ref{Sect_Prel}. 
Note that there is an additional factor $1/2$ in Sprang's theorem compared to the Ball--Rivoal theorem. 
Sprang subsequently removed this additional $1/2$ factor in unpublished work.

The proof of the Ball--Rivoal theorem involves explicit rational functions defined as follows:
\[ 
R_n(t) = n!^{s-2r} \cdot \frac{(t-rn)_{rn} \cdot (t+n+1)_{rn}}{(t)_{n+1}^{s}}, \quad n=2,4,6,\ldots, 
\]
where $r=\lfloor s/\log^2 s\rfloor$ is a positive integer, and $(x)_{k}=x(x+1)(x+2)\cdots(x+k-1)$ denotes the Pochhammer symbol. 
The quantity $S_n = D_n^{s} \cdot \sum_{\nu=1}^{\infty} R_n(\nu)$ represents a linear form in $1$ and odd zeta values with integer coefficients, where $D_n$ is the least common multiple of $1,2,\ldots,n$. 

In 2001, Zudilin \cite{Zud2001} observed that if we consider different lengths for each Pochhammer symbol appearing in $R_n(t)$, then the coefficients of the resulting linear form have a large common divisor $\Phi_n$. 
Specifically, Zudilin considered rational functions of the form:
\begin{align*}
	R_n(t) =&~ (\text{some factor})\cdot(2t+Mn) \\
	&\times \frac{\prod_{j=1}^{b}(t-rn)_{(r+\delta_{-j})n} \cdot (t+(M-\delta_{-j})n+1)_{(r+\delta_{-j})n}}{\prod_{j=1}^{s} (t+\delta_jn)_{(M-2\delta_j)n+1}}, 
\end{align*}
where $s$, $b$, $r$ and $M$ are positive integers satisfying certain parity conditions, and $\delta_{-b}$, $\ldots$, $\delta_{-1}$, $\delta_{1}$, $\ldots$, $\delta_{s}$ are nonnegative integers such that $0 \leqslant \delta_{-b} \leqslant \cdots \leqslant \delta_{-1} \leqslant \delta_{1} \leqslant \cdots \leqslant \delta_{s} < M/2$. 
Then the quantity $S_n = D_{M_1n}^{b}\prod_{j=2}^{s} D_{M_jn} \cdot \sum_{\nu=1}^{\infty} R_n^{(b-1)}(\nu)$ can be expressed as a linear form in $1$ and odd zeta values with integer coefficients. 
Here, $D_{M_jn}$ denotes the least common multiple of $1,2,\ldots,M_jn$, and $M_j= \max\{ M-2\delta_1,~M-\delta_j\}$. 
The coefficients of the linear form have a large common divisor $\Phi_n$. 
The presence of $\Phi_n$ plays a crucial role in the proof of Zudilin's remarkable theorem \cite{Zud2001}, which states that at least one of $\zeta(5)$, $\zeta(7)$, $\zeta(9)$, $\zeta(11)$ is irrational. 
It is worth noting that the irrationality of $\zeta(3)$ was proved by Ap\'ery \cite{Ap1979}, while the irrationality of $\zeta(5)$ remains an open problem\footnote{Note added in revision: a possibly correct proof by Fauzan, dated September 17, 2026, is available. See \url{https://doi.org/10.5281/zenodo.22826419}.}.

For a considerable period, it was expected that Zudilin's $\Phi_n$ factor would have a negligible impact on asymptotic results as $s \to \infty$. 
The aim of the current paper is to demonstrate that this previous expectation was inaccurate. 
In fact, we find that the $\Phi_n$ factor arising from the denominator of the rational function does have an asymptotic impact. 
By incorporating Zudilin's $\Phi_n$ factor into the original proofs of the Ball--Rivoal theorem and Sprang's theorem, we establish the following two results, which serve as the main theorems of the current paper:

\begin{theorem}
	\label{mainthm}
	For any sufficiently large even integer $s \geqslant s_0$, we have
	\[ 
	\dim_\mathbb{Q}\operatorname{Span}_{\mathbb{Q}}\left( 1,\zeta(3),\zeta(5),\ldots,\zeta(s-1) \right) \geqslant \frac{1.009}{1+\log 2}\cdot \log s.
	\]
\end{theorem}

\begin{theorem}
	\label{mainthm_p}
	Let $p$ be any prime number. 
	For any sufficiently large even integer $s \geqslant s_0(p)$, we have
	\[ 
	\dim_\mathbb{Q}\operatorname{Span}_{\mathbb{Q}}\left( 1,\zeta_p(3),\zeta_p(5),\ldots,\zeta_p(s-1) \right) \geqslant \frac{1.009}{1+\log 2}\cdot \log s.
	\]
\end{theorem}

We emphasize that the proofs of Theorems \ref{mainthm} and \ref{mainthm_p} do not rely on computer programming. 
The $\Phi_n$ factor has the asymptotic estimate $\Phi_n = \exp(\varpi n + o(n))$ as $n \to \infty$, where $\varpi$ is a computable constant. 
Typically, determining $\varpi$ requires a substantial amount of computation. 
For Theorems \ref{mainthm} and \ref{mainthm_p}, we use very simple parameters to avoid the need for computer assistance. 
On the other hand, with the help of computer-assisted calculations, we can improve the constant from $1.009$ to $1.119$ by using complicated parameters. 
We state the results as follows:

\begin{claim}
	\label{main_claim}
	For any sufficiently large even integer $s \geqslant s_0^{\prime}$, we have
	\[ 
	\dim_\mathbb{Q}\operatorname{Span}_{\mathbb{Q}}\left( 1,\zeta(3),\zeta(5),\ldots,\zeta(s-1) \right) \geqslant \frac{1.119356}{1+\log 2}\cdot \log s.
	\]
	For any prime number $p$, for any sufficiently large even integer $s \geqslant s_0^{\prime}(p)$, we have
	\[ 
	\dim_\mathbb{Q}\operatorname{Span}_{\mathbb{Q}}\left( 1,\zeta_p(3),\zeta_p(5),\ldots,\zeta_p(s-1) \right) \geqslant \frac{1.119356}{1+\log 2}\cdot \log s.
	\]
\end{claim}
Note that $1/(1+\log 2) = 0.590\ldots$ and $1.119356/(1+\log 2) = 0.661\ldots\,$. 

Meanwhile, we also obtain the following result as a byproduct:
\begin{claim}
	\label{thm75}
	We have
	\[ 
	\dim_\mathbb{Q}\operatorname{Span}_{\mathbb{Q}}\left( 1,\zeta(3),\zeta(5),\ldots,\zeta(75) \right) \geqslant 3. 
	\]
\end{claim}

The proofs of Claims \ref{main_claim} and \ref{thm75} involve time-consuming calculations (by computer programming). 
It is unreasonable to require the referee to check the correctness of these time-consuming calculations, as the main ideas of this paper are already conveyed by Theorem \ref{mainthm} and Theorem \ref{mainthm_p}. 
Therefore, we state Claims \ref{main_claim} and \ref{thm75} as claims rather than theorems. 
However, we will mention other simple parameters in \S \ref{Sect_comp} that improve the constant $1.009$ in Theorems \ref{mainthm} and \ref{mainthm_p}. 

We would like to point out that Theorem \ref{mainthm} is subsumed by a much stronger result of Fischler, as stated below (but Theorem \ref{mainthm_p} is new):

\begin{theorem}[Fischler \cite{Fis2021+}, 2026]
	For any sufficiently large odd integer $s$, we have
	\[ 
	\dim_\mathbb{Q}\operatorname{Span}_{\mathbb{Q}}\left( 1,\zeta(3),\zeta(5),\ldots,\zeta(s) \right) \geqslant 0.21 \cdot \sqrt{\frac{s}{\log s}}. 
	\]
\end{theorem}
One notable aspect of Fischler's work is the use of non-explicit rational functions. 
Because these rational functions are not explicit, proving a certain nonvanishing property becomes challenging. 
Fischler overcame this difficulty by employing a generalized version of Shidlovsky's lemma. 
In contrast, our proof of Theorem \ref{mainthm} relies on explicit rational functions, making it straightforward to establish the required nonvanishing property. 
For the $p$-adic case, our proof of Theorem \ref{mainthm_p} also relies on explicit rational functions. 
However, the corresponding nonvanishing property is not obvious and requires a more delicate proof.

Although Theorem \ref{mainthm} is subsumed by Fischler's result, we present its proof in this paper to highlight the parallels between the classical case and the $p$-adic case.

\subsection{Notation}

Throughout this paper, the letters $p$ and $q$ always denote prime numbers. 
As usual, the symbols $\mathbb{Z}_p$, $\mathbb{Q}_p$, and $\mathbb{C}_p$ denote the ring of $p$-adic integers, the field of $p$-adic numbers, and the completion of an algebraic closure of $\mathbb{Q}_p$, respectively. 
We use $v_p(x)$ to denote the $p$-adic order of $x$, normalized by $v_p(p)=1$. 
The $p$-adic norm is defined by $|x|_p = p^{-v_p(x)}$. 
The notation $|y|$ denotes the Archimedean norm of $y$. 
The function $\log(\cdot)$ always denotes the natural logarithm (with respect to the base $e = 2.718\ldots$). 
The function $\log_p(\cdot)$ denotes the $p$-adic logarithm, which is defined by
\[ 
\log_p(1+x) := \sum_{k=1}^{\infty} \frac{(-1)^{k-1}}{k} x^k \quad \text{for}\ x \in \mathbb{C}_p \text{~with~} |x|_p < 1. 
\]

We use the standard asymptotic notations $O$, $o$, and $\sim$. 
For two real-valued functions $f(x)$ and $g(x)$ with $g(x) > 0$, the expressions $f(x)=O(g(x))$, $f(x)=o(g(x))$, and $f(x) \sim g(x)$ mean that $|f(x)| \leqslant Cg(x)$ for some constant $C>0$, $\lim f(x)/g(x) =0$, and $\lim f(x)/g(x) = 1$, respectively.

For any positive integer $m$, we use $D_m$ to denote the least common multiple of $1,2,\ldots,m$. 
The Pochhammer symbol $(x)_m$ is defined by 
\[ 
(x)_m := x(x+1)\cdots(x+m-1), 
\]
with the convention $(x)_0 := 1$. 
The Euler totient function is denoted by $\varphi(\cdot)$. 
The notation $\phi(\cdot)$ is used to denote some other function in this paper (see \eqref{defi_phi}).
\subsection{An outline of the proof of Theorem \ref{mainthm}}

In this subsection, we provide a brief outline of the proof of Theorem \ref{mainthm}. 
The reader who is familiar with the work of Ball and Rivoal, and of Zudilin will quickly recognize that our proof follows by incorporating Zudilin's $\Phi_n$ factor \cite{Zud2001} into the original proof of Ball--Rivoal \cite{BR2001}.

Let $s$ be a sufficiently large even integer, and take $r = \lfloor s/\log^2 s \rfloor$. 
Consider the rational functions:
\begin{align*}
	R_n(t) = &\frac{(6n)!^{s/2}(4n)!^{s/2}}{n!^{2r}} \cdot (t+3n) \notag\\
	& \times \frac{(t-rn)_{rn} \cdot (t+6n+1)_{rn}}{ (t)_{6n+1}^{s/2} \cdot (t+n)_{4n+1}^{s/2} }, \quad n=1,2,3,\ldots. 
\end{align*}

Using arguments similar to those in \cite{BR2001} and \cite{Zud2004}, we have the following facts:
\begin{itemize}
	\item The quantity $S_n = \sum_{\nu=1}^{\infty} R_n(\nu)$ can be expressed as a linear form in $1$ and odd zeta values with rational coefficients:
	\[ 
	S_n = \rho_0 + \sum_{3 \leqslant i \leqslant s-1 \atop i \text{~odd~}} \rho_i \zeta(i). 
	\]
	
	\item We have $\Phi_n^{-s/2}D_{6n}^s \cdot \rho_i \in \mathbb{Z}$ for any $i=0,3,5,\ldots,s-1$, provided that $n>s^2$. 
	Here $\Phi_n$ is a product of certain primes: 
	\[ 
	\Phi_n =  \prod_{\sqrt{6n} < q \leqslant 6n \atop q \text{~prime~}} q^{\phi(n/q)}, 
	\]
	with the function
	\begin{align*}
		\phi(x) &= \inf_{y \in \mathbb{R}} \left(  \lfloor6x\rfloor - \lfloor y \rfloor - \lfloor 6x - y \rfloor + \lfloor4x\rfloor - \lfloor y- x \rfloor - \lfloor 5x - y \rfloor \right) \\
		&= \begin{cases}
			1 \quad\text{if~} \{x\} \in \left[\frac{1}{6},\frac{1}{5}\right) \cup \left[\frac{1}{3},\frac{2}{5}\right) \cup \left[\frac{1}{2},\frac{3}{5}\right) \cup \left[\frac{3}{4},\frac{4}{5}\right), \\
			0 \quad\text{otherwise.} 
		\end{cases}
	\end{align*}
	(As usual, $\{x\}$ denotes the fractional part of $x$.) 
	We remark that $\Phi_n$ receives contributions only from the `denominator type elementary bricks' of the rational function $R_n(t)$.
	
	\item We have the following bound for the coefficients $\rho_i$ of the linear forms: 
	\[
	\max_{i=0,3,5,\ldots,s-1} |\rho_i| \leqslant \exp\left( \beta(s)n + o(n) \right) \quad \text{as~} n \to \infty, 
	\]
	where $\beta(s)$ is a constant depending only on $s$, and
	\[ 
	\beta(s) \sim (5\log 2) \cdot s  \quad\text{as~} s \to \infty. 
	\] 
	
	\item We have the following asymptotic estimate for the linear forms:
	\[
	S_n = \exp\left( -\alpha(s) n + o(n)  \right) \quad\text{as~} n \to \infty,
	\] 
	where $\alpha(s)$ is a constant depending only on $s$, and 
	\[
	\alpha(s) \sim 5 \cdot s\log s \quad \text{as~} s \to \infty. 
	\]
	
	\item Let $\psi(x) = \Gamma^{\prime}(x)/\Gamma(x)$ denote the digamma function. We have the following asymptotic estimate for the $\Phi_n$ factor:
	\[
	\Phi_n = \exp\left( \varpi n + o(n) \right) \quad\text{as~} n \to \infty,
	\]
	where 
	\begin{align*}
		\varpi =&~ \psi\left(\frac{1}{5}\right) - \psi\left(\frac{1}{6}\right) + \psi\left(\frac{2}{5}\right) - \psi\left(\frac{1}{3}\right) \\
		&+ \psi\left(\frac{3}{5}\right) - \psi\left(\frac{1}{2}\right) + \psi\left(\frac{4}{5}\right) - \psi\left(\frac{3}{4}\right) \\
		=& 2.157479\ldots.
	\end{align*}
\end{itemize}

Applying Nesterenko's linear independence criterion to the sequence of linear forms $ \{ \Phi_{n}^{-s/2}D_{6n}^s \cdot S_n\}_{n > s^2}$, we obtain
\begin{align*}
	&\dim_{\mathbb{Q}}\operatorname{Span}_{\mathbb{Q}}\left( 1,\zeta(3),\zeta(5),\ldots,\zeta(s-1) \right) \\
	\geqslant&~ 1 + \frac{\alpha(s) - 6s + \varpi s/2}{\beta(s) + 6s - \varpi s/2} \\
	=&~ (C - o(1)) \cdot \log s \quad \text{as~} s \to \infty,
\end{align*}
where
\begin{align*}
	C = \frac{5}{5\log 2 + 6-\varpi/2} = \frac{1.009388\ldots}{1+\log 2}.
\end{align*}
Therefore, we have $\dim_{\mathbb{Q}}\operatorname{Span}_{\mathbb{Q}}\left( 1,\zeta(3),\zeta(5),\ldots,\zeta(s-1) \right) \geqslant (1.009\cdot \log s)/(1+\log 2)$ for any sufficiently large even integer $s$. 

To improve the constant from $1.009$ to $1.119$, we will consider more complicated rational functions $R_n(t)$, as defined in \eqref{defi_R_n(t)}.

\subsection{Similarities between the $p$-adic case and the classical case}

In this subsection, we briefly discuss the similarities (and differences) between the proofs of Theorem \ref{mainthm} and Theorem \ref{mainthm_p}. 

In the classical case, using rational functions to construct linear forms in $1$ and zeta values is based on the observation
\begin{equation}
	\label{rtol}
	\sum_{t=1}^{\infty} \frac{1}{(t+k)^i} = \zeta(i) - \sum_{\ell=1}^{k} \frac{1}{\ell^i} \quad (k \in \mathbb{Z}_{>0} \ \text{and}\ i \in \mathbb{Z}_{\geqslant 2}).
\end{equation} 
We have no straightforward $p$-adic analogue of \eqref{rtol}. 
Instead, we need an intermediate step involving $p$-adic Hurwitz zeta values. 
Recall that the classical Hurwitz zeta function is defined by $\zeta(i,x) := \sum_{t = 0}^{\infty} (t+x)^{-i}$. 
We have
\[ 
\sum_{t=0}^{\infty} \frac{1}{(t+k+x)^i} = \zeta(i,x) - \sum_{\ell=0}^{k-1} \frac{1}{(\ell+x)^i} \quad(x \in \mathbb{R}_{>0}). 
\]
Its $p$-adic analogue is the following evaluation of the Volkenborn integral:
\begin{align*}
	&~ -\int_{\mathbb{Z}_p} \frac{\mathrm{d}t}{(1-i)(t+k+x)^{i-1}} \\
	=&~ \omega(x)^{1-i}\zeta_p(i,x) - \sum_{\ell=0}^{k-1} \frac{1}{(\ell+x)^i} \quad (x \in \mathbb{Q}_{p} \text{~with~} |x|_p \geqslant q_p),
\end{align*}
where $\omega(x)$ is the Teichm\"uller character and $\zeta_p(i,x)$ is the $p$-adic Hurwitz zeta function.

Note that
\[ 
\frac{\mathrm{d}}{\mathrm{d}t}\left( \frac{1}{(1-i)(t+k+x)^{i-1}} \right) = \frac{1}{(t+k+x)^i}. 
\]
It turns out that, for a rational function $R(t)$ of the form
\[ 
R(t) =\sum_{i=1}^{s}\sum_{k=0}^{n} \frac{a_{i,k}}{(t+k)^i} \quad (n \in \mathbb{Z}_{>0},\ s \in \mathbb{Z}_{>0},\ \text{and}\ a_{i,k} \in \mathbb{Q}), 
\] 
we have
\begin{align}
	\sum_{t=0}^{\infty} \frac{\mathrm{d}}{\mathrm{d}t}R(t+x) &= \rho_0 + \sum_{i=2}^{s+1} \rho_i \zeta(i,x) \quad (x \in \mathbb{R}_{>0}), \label{cla}\\
	-\int_{\mathbb{Z}_p} R(t+x)\,\mathrm{d}t &= \rho_0 + \sum_{i=2}^{s+1} \rho_i \omega(x)^{1-i}\zeta_p(i,x) \quad (x \in \mathbb{Q}_p \text{~with~} |x|_p \geqslant q_p). \label{pad}
\end{align} 
If $x \in \mathbb{Q}_{>0}$ with $|x|_p \geqslant q_p$, then the coefficients $\rho_0, \rho_i \in \mathbb{Q}$ in \eqref{cla} and \eqref{pad} are exactly the same. 
This is the main similarity between the classical case and the $p$-adic case.

Equations \eqref{cla} and \eqref{pad} suggest that the $p$-adic analogue of 
\[ 
\sum_{t=0}^{\infty} R(t+x) 
\]
should be
\[ 
-\int_{\mathbb{Z}_p} W(t+x)\,\mathrm{d}t 
\]
for some suitable primitive function $W(t)$ (that is, $W'(t) = R(t)$). 

However, to prove Theorem \ref{mainthm} and Theorem \ref{mainthm_p}, we need to use different sequences of rational functions $\{ R_n(t) \}_{n \geqslant 1}$ and $\{ V_n(t) \}_{n \geqslant 1}$ (and the primitive functions $\{ W_n(t)\}_{n \geqslant 1}$, $W^{\prime}_n(t)=V_n(t)$), respectively. 
Let us briefly explain the reason. 
Recall that the original rational functions used by Ball--Rivoal \cite{BR2001} are of the form
\begin{equation}
	\label{BR-R}
	R_n(t) = (\text{some factor}) \cdot \frac{(t-rn)_{rn}(t+n+1)_{rn}}{(t)_{n+1}^s}.
\end{equation}
The `left factor' $(t-rn)_{rn}$ is designed to ensure that $\sum_{t=1}^{\infty} R_n(t)$ is small. 
The `right factor' $(t+n+1)_{rn}$ is designed for symmetry to eliminate all even zeta values. 
In the $p$-adic case, we do not need the symmetry because $\zeta_p(i) = 0$ for any positive even integer $i$. 
The candidates for the $p$-adic case are
\[ 
V_n(t) = (\text{some factor}) \cdot \frac{\prod_{1 \leqslant \nu \leqslant p^l \atop p \nmid \nu} (t+\nu/p^l)_{n}}{(t)_{n+1}^s}, 
\]
where $l$ is a large positive integer. 
The role of $l$ is similar to that of $r$ in \eqref{BR-R}: it makes
\[ 
-\int_{\mathbb{Z}_p} W_n(t+x)\,\mathrm{d}t \quad \left(x = \frac{\nu}{p^l},~1\leqslant \nu \leqslant p^l,~ p \nmid \nu \right) 
\] 
$p$-adically small. 
We need the `middle factor'
\[ 
\prod_{1 \leqslant \nu \leqslant p^l \atop p \nmid \nu} \left(t+\frac{\nu}{p^l}\right)_{n}  
\] 
of $V_n(t)$ to control the denominator of the coefficient $\rho_0$ (that is, the coefficient corresponding to $1$ in the resulting linear combination of $1$ and $p$-adic odd zeta values). 
This use of the `middle factor' also appears in \cite{FSZ2019,LS2023+,LY2020}. 
However, to establish a certain nonvanishing property for the $p$-adic case, we need a few technical modifications to $V_n(t)$ (see \eqref{defi_V_n(t)}).

\subsection{Structure of the paper}

In Section \ref{Sect_line}, we introduce Nesterenko's linear independence criterion and its $p$-adic variant. 

Sections \ref{Sect_sett}--\ref{Sect_proo} are devoted to Theorem \ref{mainthm}. 
In Section \ref{Sect_sett}, we set up parameters and construct linear forms using explicit rational functions. 
In Section \ref{Sect_arit}, we prove the arithmetic properties of the linear forms. 
In Section \ref{Sect_asym}, we study the asymptotic behavior of these linear forms. 
Finally, in Section \ref{Sect_proo}, we prove Theorem \ref{mainthm}. 

Sections \ref{Sect_Prel}--\ref{Sect_proo_p} are devoted to Theorem \ref{mainthm_p}. 
In Section \ref{Sect_Prel}, we recall the definition of $p$-adic zeta values. 
In Section \ref{Sect_sett_p}, we set up parameters and construct linear forms using explicit rational functions. 
In Section \ref{Sect_arit_p}, we prove the arithmetic properties of the linear forms. 
In Section \ref{Sect_arch_est}, we estimate the Archimedean norm of the coefficients of these linear forms. 
In Section \ref{Sect_p_adic_est}, we estimate the $p$-adic norm of these linear forms. 
Finally, in Section \ref{Sect_proo_p}, we prove Theorem \ref{mainthm_p}.

In Section \ref{Sect_comp}, we focus on the computational aspect of our method. 
We prove Claim \ref{main_claim} and Claim \ref{thm75} in this section.

\section{Linear independence criteria}
\label{Sect_line}

Two fundamental criteria for linear independence are Siegel's criterion \cite{Sie1929} and Nesterenko's criterion \cite{Nes1985}. 
Each criterion has its own advantages. 
Fischler’s construction \cite{Fis2021+} is suited to Siegel’s criterion, whereas Nesterenko’s criterion is convenient for the explicit constructions of Rivoal \cite{Riv2000} and Ball–Rivoal \cite{BR2001}.
In the current paper, we will utilize Nesterenko's criterion and its $p$-adic variant to prove Theorem \ref{mainthm} and Theorem \ref{mainthm_p}, respectively.

\begin{theorem}[Nesterenko's linear independence criterion \cite{Nes1985}, 1985]
	\label{Nes}
	Let $\xi_1$, $\xi_2$, $\ldots$, $\xi_s$ be real numbers, with $s \in \mathbb{Z}_{>0}$. 
	Let $\alpha, \beta$ be two positive constants. 
	For any $n \in \mathbb{Z}_{>0}$, let $L_n(X_0,X_1,\ldots,X_s) = \sum_{i=0}^{s} l_{n,i}X_i$ be a linear form with integer coefficients $l_{n,i} \in \mathbb{Z}$ for $i=0,1,\ldots,s$. 
	Suppose the following conditions hold\textup{:}
	\begin{itemize}
		\item $|L_n(1,\xi_1,\xi_2,\ldots,\xi_s)| = \exp\left( -\alpha n + o(n)  \right)$ as $n \to \infty$\textup{;}
		\item $\max_{0\leqslant i \leqslant s} |l_{n,i}| \leqslant \exp\left( \beta n + o(n) \right)$ as $n \to \infty$.
	\end{itemize}
	Then, we have 
	\[ 
	\dim_{\mathbb{Q}}\operatorname{Span}_{\mathbb{Q}}\left( 1,\xi_1,\xi_2,\ldots,\xi_s \right) \geqslant 1 +\frac{\alpha}{\beta}. 
	\]
\end{theorem}

In 2010, Fischler and Zudilin \cite{FZ2010} refined Nesterenko's criterion. 
We present their result below. 
Although this refined version is not necessary for the proof of Theorem \ref{mainthm}, it will be utilized in \S \ref{Sect_comp} to prove that $\dim_{\mathbb{Q}}\operatorname{Span}_{\mathbb{Q}}\left( 1,\zeta(3),\zeta(5),\ldots,\zeta(75) \right) \geqslant 3$. 

\begin{theorem}[Fischler--Zudilin \cite{FZ2010}, 2010]
	\label{reNes}
	Retain the notation and hypotheses in Theorem \textup{\ref{Nes}}. 
	Suppose, in addition, that for each $n \in \mathbb{Z}_{>0}$ and each $i=1,2,\ldots,s$, the integer coefficient $l_{n,i}$ has a positive divisor $d_{n,i}$ such that\textup{:}
	\begin{itemize}
		\item $d_{n,i}$ divides $d_{n,i+1}$ for any $n \in \mathbb{Z}_{>0}$ and any $i=1,2,\ldots,s-1$\textup{;}
		\item $d_{n,j}/d_{n,i}$ divides $d_{n+1,j}/d_{n+1,i}$ for any $n \in \mathbb{Z}_{>0}$ and any $0 \leqslant i < j \leqslant s$, with $d_{n,0}=1$\textup{;}
		\item For each $i=1,2,\ldots,s$, we have $d_{n,i}= \exp\left( \gamma_i n + o(n) \right)$ as $n \to \infty$ for some constant $\gamma_i \geqslant 0$.
	\end{itemize}
	Let $d = \dim_{\mathbb{Q}}\operatorname{Span}_{\mathbb{Q}}\left( 1,\xi_1,\xi_2,\ldots,\xi_s \right)$. 
	Then, we have
	\[ 
	d \geqslant 1 + \frac{\alpha + \gamma_1 + \gamma_2 + \cdots + \gamma_{d-1}} {\beta}. 
	\]
\end{theorem}

In 2012, Nesterenko \cite{Nes2012} proved a $p$-adic version of his linear independence criterion. 
We use the following formulation (see \cite[Theorem 1.4]{Spr2020}):

\begin{theorem}[Nesterenko's $p$-adic linear independence criterion \cite{Nes2012}, 2012]
	\label{Nes_p}
	Let $p$ be any prime number. 
	Let $\xi_1,\xi_2,\ldots,\xi_s \in \mathbb{Q}_p$ be $p$-adic numbers, with $s \in \mathbb{Z}_{>0}$. 
	Let $\alpha_1, \alpha_2, \beta$ be three positive real constants. 
	For any $n \in \mathbb{Z}_{>0}$, let $L_n(X_0,X_1,\ldots,X_s) = \sum_{i=0}^{s} l_{n,i}X_i$ be a linear form with integer coefficients $l_{n,i} \in \mathbb{Z}$ for $i=0,1,\ldots,s$. 
	Suppose the following conditions hold\textup{:}
	\begin{itemize}
		\item $\exp\left( -\alpha_1 n + o(n)  \right) \leqslant |L_n(1,\xi_1,\xi_2,\ldots,\xi_s)|_p \leqslant \exp\left( -\alpha_2 n + o(n)  \right)$ as $n \to \infty$\textup{;}
		\item $\max_{0\leqslant i \leqslant s} |l_{n,i}| \leqslant \exp\left( \beta n + o(n) \right)$ as $n \to \infty$.
	\end{itemize}
	Then, we have 
	\[ 
	\dim_{\mathbb{Q}}\operatorname{Span}_{\mathbb{Q}}\left( 1,\xi_1,\xi_2,\ldots,\xi_s \right) \geqslant \frac{\alpha_1}{\beta+\alpha_1-\alpha_2}. 
	\]
\end{theorem}

\section{Setup for the proof for the classical case}
\label{Sect_sett}

In this section, we will first set up the parameters. 
Then, we will construct a sequence of rational functions $R_n(t)$ and utilize them to produce linear forms in $1$ and odd zeta values. 

\subsection{Parameters}

Fix a positive integer $M$ and a finite collection of nonnegative integers $\{\delta_j\}_{j=1}^{J}$ such that
\begin{equation*}
	0 \leqslant \delta_1 \leqslant \delta_2 \leqslant \cdots \leqslant \delta_J < \frac{M}{2},
\end{equation*}
where $J \in \mathbb{Z}_{>0}$ denotes the number of $\delta_j$'s. 

Let $s$ be a sufficiently large multiple of $2J$. 
Define the `numerator length parameter'
\begin{equation*}
	r := \left\lfloor \frac{s}{\log^2 s} \right\rfloor
\end{equation*}
to be the same as in the Ball--Rivoal setting \cite{BR2001}.

\subsection{Rational functions}

Recall that $(x)_k=x(x+1)\cdots(x+k-1)$ denotes the Pochhammer symbol. 
For any $n \in \mathbb{Z}_{>0}$, define the following rational function $R_n(t)$:
\begin{align}
	R_n(t) := &\frac{\prod_{j=1}^{J} ((M-2\delta_j)n)!^{s/J}}{n!^{2r}} \cdot (2t+Mn) \notag\\
	& \times \frac{(t-rn)_{rn} (t+Mn+1)_{rn}}{\prod_{j=1}^{J} (t+\delta_jn)_{(M-2\delta_j)n+1}^{s/J}}. \label{defi_R_n(t)}
\end{align} 
Since $s$ is a multiple of $2J$, this rational function has rational coefficients, that is, $R_n(t) \in \mathbb{Q}(t)$, and it exhibits the following symmetry:
\begin{equation}
	\label{symmetry}
	R_n(-t-Mn) = - R_n(t). 
\end{equation}
Note that the degree of the rational function $R_n(t)$ is
\[ 
\deg R_n  = 1 + 2rn - \frac{s}{J}\sum_{j=1}^{J} \left((M-2\delta_j)n+1\right). 
\]
Since $r = \lfloor s/\log^2 s \rfloor$ and $s$ is assumed to be sufficiently large, we have 
\begin{equation}
	\label{degleq-2}
	\deg R_n(t) \leqslant -2 \quad\text{for any~}n \in \mathbb{Z}_{>0}.
\end{equation}
Thus, the rational function $R_n(t)$ has the following partial-fraction decomposition:
\begin{equation}
	\label{defi_anik}
	R_n(t) =: \sum_{i=1}^{s}\sum_{k=\delta_1 n}^{(M-\delta_1)n} \frac{a_{n,i,k}}{(t+k)^i},
\end{equation} 
where the coefficients $a_{n,i,k} \in \mathbb{Q}$ (for $1 \leqslant i \leqslant s$ and $\delta_1n \leqslant k \leqslant (M-\delta_1)n$) are uniquely determined by $R_n(t)$. 

\subsection{Linear forms}

We define
\begin{align}
	\rho_{n,i} &:= \sum_{k=\delta_1n}^{(M-\delta_1)n} a_{n,i,k} \quad (i=1,2,3,\ldots,s), \label{defi_rho_i} \\
	\rho_{n,0} &:= -\sum_{i=1}^{s}\sum_{k=\delta_1n}^{(M-\delta_1)n}\sum_{\ell=1}^{k} \frac{a_{n,i,k}}{\ell^i}. \label{defi_rho_0}
\end{align}
Note that the symmetry of the rational function (Equation \eqref{symmetry}) implies the symmetry of its coefficients in the partial-fraction decomposition; that is
\[ 
a_{n,i,k} = (-1)^{i+1} a_{n,i,Mn-k} \quad (1 \leqslant i \leqslant s \ \text{and}\ \delta_1n \leqslant k \leqslant (M-\delta_1)n).  
\]
Therefore, we have 
\[ 
\rho_{n,i} = 0 \quad\text{for~} i=2,4,6,\ldots,s.
\] 
On the other hand, by \eqref{degleq-2} and \eqref{defi_anik}, we have
\[ 
\rho_{n,1} = \lim_{t \to \infty} tR_n(t) = 0. 
\]

Define 
\begin{equation}
	\label{defi_Sn}
	S_n := \sum_{\nu=1}^{\infty} R_n(\nu). 
\end{equation}
Using arguments similar to those in \cite[Lemme 1]{BR2001}, we have the following lemma:
\begin{lemma}
	\label{lem_linearforms}
	For any $n \in \mathbb{Z}_{>0}$, $S_n$ is a linear form in $1$ and odd zeta values, given by
	\[ 
	S_n = \rho_{n,0} + \sum_{3 \leqslant i \leqslant s-1 \atop i \text{~odd}} \rho_{n,i} \zeta(i),
	\]
	where the coefficients $\rho_{n,i}$ \textup{(}for $i=3,5,\ldots,s-1$\textup{)} are defined by \eqref{defi_rho_i}, and $\rho_{n,0}$ is defined by \eqref{defi_rho_0}.
\end{lemma}

\section{Arithmetic properties for the classical case}
\label{Sect_arit}
In this section, we will first introduce two types of `elementary bricks': the `denominator type elementary brick' (see Lemma \ref{lem_G}) and the `numerator type elementary brick' (see Lemma \ref{lem_F}). 
Our rational functions in \eqref{defi_R_n(t)} and many other rational functions in the literature (for example, \cite{Riv2000,Zud2001}) are finite products of these `elementary bricks'. 
Then, we will study the denominators of the coefficients $a_{n,i,k}$ (defined by \eqref{defi_anik}) and $\rho_{n,i}$ (defined by \eqref{defi_rho_0} and \eqref{defi_rho_i}).

Recall that we denote by
\[ 
D_m = \operatorname{lcm}[1,2,3,\ldots,m] 
\]
the least common multiple of the numbers $1,2,3,\ldots,m$ for any positive integer $m$. 
We denote by $f^{(\lambda)}(t)$ the $\lambda$th derivative of a function $f(t)$ for any nonnegative integer $\lambda$. 
For a prime number $p$, the notation $v_p(x)$ denotes the $p$-adic order of $x$. 

\begin{lemma}
	\label{lem_G}
	Let $a,b,a_0,b_0$ be integers such that $a_0 \leqslant a \leqslant b \leqslant b_0$ and $b_0 > a_0$. 
	Consider the `denominator type elementary brick'
	\[ 
	G(t) = \frac{(b-a)!}{(t+a)_{b-a+1}}. 
	\]
	Then, we have
	\begin{equation}
		\label{lem_G_1}
		D_{b_0-a_0}^{\lambda} \cdot \frac{1}{\lambda!} \left( G(t)(t+k) \right)^{(\lambda)} \big|_{t=-k} \in \mathbb{Z}
	\end{equation} 
	for any $k \in [a_0,b_0]\cap \mathbb{Z}$ and any $\lambda \in \mathbb{Z}_{\geqslant 0}$. 
	
	Moreover, for any prime number $q > \sqrt{b_0-a_0}$, any $k \in [a_0,b_0]\cap \mathbb{Z}$, and any $\lambda \in \mathbb{Z}_{\geqslant 0}$, we have
	\begin{equation}
		\label{lem_G_2}
		v_q\left( \left( G(t)(t+k) \right)^{(\lambda)} \big|_{t=-k} \right) \geqslant -\lambda + \left\lfloor \frac{b-a}{q} \right\rfloor - \left\lfloor \frac{k-a}{q} \right\rfloor - \left\lfloor \frac{b-k}{q} \right\rfloor. 
	\end{equation}
\end{lemma}

\begin{proof}
	See \cite[Lemmas 16 and 18]{Zud2004}. 
	We have replaced $b$ and $b_0$ in \cite[Lemmas 16 and 18]{Zud2004} by $b+1$ and $b_0+1$, respectively. 
\end{proof}

\begin{lemma}
	\label{lem_F}
	Let $a,b,c,m$ be integers with $m > 0$ and $b>0$. 
	Consider the `numerator type elementary brick'
	\[ 
	F(t) =  b^{m} \prod_{q \mid b \atop q \text{~prime~}} q^{\lfloor m/(q-1) \rfloor} \cdot \frac{(ct+a/b)_m}{m!}. 
	\]
	Then, we have
	\[
	D_{m}^{\lambda} \cdot \frac{1}{\lambda!}  F^{(\lambda)}(t)  \big|_{t=-k} \in \mathbb{Z} 
	\]
	for any $k \in \mathbb{Z}$ and any $\lambda \in \mathbb{Z}_{\geqslant 0}$.
\end{lemma}

Lemma \ref{lem_F} is known in the literature but not stated as above. 
For example, one can slightly modify the arguments in \cite[Proposition 3.2]{LY2020} to obtain a proof. 
For completeness, we give the proof.

\begin{proof}[Proof of Lemma \textup{\ref{lem_F}}]
	Fix $k \in \mathbb{Z}$ and $\lambda \in \mathbb{Z}_{\geqslant 0}$. 
	Note that 
	\[ 
	F(t-k) =  \frac{\prod_{q \mid b \atop q ~\text{prime}} q^{\lfloor m/(q-1) \rfloor}}{m!} \cdot \prod_{\nu=0}^{m-1} (bct -bck +b\nu + a). 
	\]
	Write
	\begin{equation}
		\label{351}
		\prod_{\nu=0}^{m-1} (bct-bck+b\nu + a) =  \sum_{j=0}^{m} C_jt^{j}, \quad\text{where~} C_j \in \mathbb{Z} \text{~for~} j=0,1,\ldots,m.
	\end{equation} 
	Then, we have
	\[ 
	\frac{1}{\lambda!}  F^{(\lambda)}(t)  \big|_{t=-k} = \frac{\prod_{q \mid b \atop q ~\text{prime}} q^{\lfloor m/(q-1) \rfloor}}{m!} \cdot C_{\lambda}, 
	\]
	with the convention $C_{\lambda} = 0$ for $\lambda > m$.
	
	It remains to prove that, for any prime $q_0$, we have
	\begin{equation}
		\label{352}
		v_{q_0} \left( D_m^{\lambda}  \frac{\prod_{q \mid b \atop q ~\text{prime}} q^{\lfloor m/(q-1) \rfloor}}{m!} C_{\lambda}\right)  \geqslant 0. 
	\end{equation}
	Since $C_\lambda \in \mathbb{Z}$ and $v_{q_0}(m!) \leqslant \lfloor  m/(q_0-1)\rfloor$, Inequality \eqref{352} holds for any prime $q_0 \mid b$. 
	Now, assume $q_0 \nmid b$. 
	By expanding the left-hand side of \eqref{351}, we see that $C_\lambda$ can be expressed as a sum of finitely many terms of the form
	\[ 
	{(bc)}^{\lambda} \prod_{i=1}^{\lambda+1} U_i, 
	\]
	where each $U_i$ is a product of $\ell_i \in \mathbb{Z}_{\geqslant 0}$ consecutive terms in an integer arithmetic progression with common difference $b$ such that $\sum_{i=1}^{\lambda+1} \ell_i = m-\lambda$. 
	Since $q_0 \nmid b$, we have
	\[ 
	v_{q_0}(U_i) \geqslant \sum_{j=1}^{\infty} \left\lfloor  \frac{\ell_i}{q_0^{j}}\right\rfloor, 
	\]
	and hence
	\[ 
	v_{q_0} \left( {(bc)}^{\lambda} \prod_{i=1}^{\lambda+1} U_i \right) \geqslant \sum_{j=1}^{\infty} \sum_{i=1}^{\lambda+1} \left\lfloor  \frac{\ell_i}{q_0^{j}}\right\rfloor.  
	\]
	Noting that
	\[ 
	\sum_{i=1}^{\lambda+1} \left\lfloor  \frac{\ell_i}{q_0^{j}}\right\rfloor \geqslant \sum_{i=1}^{\lambda+1} \frac{\ell_i - (q_0^{j}-1)}{q_0^{j}} = \frac{m+1}{q_0^{j}} - \lambda - 1 > \left\lfloor \frac{m}{q_0^{j}} \right\rfloor - \lambda - 1  
	\]
	and the left-hand side above is an integer, we have
	\[ 
	\sum_{i=1}^{\lambda+1} \left\lfloor  \frac{\ell_i}{q_0^{j}}\right\rfloor \geqslant \left\lfloor \frac{m}{q_0^{j}} \right\rfloor - \lambda. 
	\]
	Thus, 
	\begin{align*}
		v_{q_0} \left( {(bc)}^{\lambda} \prod_{i=1}^{\lambda+1} U_i \right) &\geqslant \sum_{j=1}^{\infty} \sum_{i=1}^{\lambda+1} \left\lfloor  \frac{\ell_i}{q_0^{j}}\right\rfloor = \sum_{j=1}^{\lfloor \log m / \log q_0 \rfloor} \sum_{i=1}^{\lambda+1} \left\lfloor  \frac{\ell_i}{q_0^{j}}\right\rfloor \\
		&\geqslant \sum_{j=1}^{\lfloor \log m / \log q_0 \rfloor} \left( \left\lfloor \frac{m}{q_0^{j}} \right\rfloor - \lambda \right) = v_{q_0}(m!) - \lambda v_{q_0}(D_m).
	\end{align*}
	Therefore, we have $v_{q_0}(C_{\lambda}) \geqslant v_{q_0}(m!) - \lambda v_{q_0}(D_m)$, which proves Inequality \eqref{352}.
\end{proof}

Now, we study the arithmetic property of the coefficients $a_{n,i,k}$ defined by \eqref{defi_anik}. 
The proof of the following lemma is essentially the same as (part of) the proof of \cite[Lemma 19]{Zud2004}, but we would like to present the details here.

\begin{lemma}
	\label{lem_arith_anik}
	For any integer $n > s^2$, we have
	\[ 
	\Phi_n^{-s/J} D_{(M-2\delta_1)n}^{s-i} \cdot  a_{n,i,k} \in \mathbb{Z} \quad (1 \leqslant i \leqslant s \ \text{and}\ \delta_1n \leqslant k \leqslant (M-\delta_1)n),  
	\]
	where the factor $\Phi_n$ is a product of certain primes\textup{:}
	\begin{equation}
		\label{defi_Phi}
		\Phi_n := \prod_{\sqrt{Mn} < q \leqslant (M-2\delta_1)n \atop q \text{~prime}} q^{\phi(n/q)}, 
	\end{equation}
	and the function $\phi(\cdot)$ is defined by
	\begin{equation}
		\label{defi_phi}
		\phi(x) := \inf_{y \in \mathbb{R}} \sum_{j=1}^{J}  \left( \lfloor(M-2\delta_j)x\rfloor - \lfloor y-\delta_j x \rfloor - \lfloor (M-\delta_j)x - y \rfloor \right).
	\end{equation}
\end{lemma}

\begin{proof}
	Fix any $i \in \{1,2,\ldots,s\}$ and any $k \in [\delta_1n,~(M-\delta_1)n] \cap \mathbb{Z}$.
	
	By \eqref{defi_anik}, we have
	\begin{equation}\label{a_i_k_determined_by_R_n}
		a_{n,i,k} = \frac{1}{(s-i)!} \left( R_n(t)(t+k)^{s} \right)^{(s-i)} \Big|_{t=-k}.
	\end{equation}
	
	Consider the following elementary bricks:
	\begin{align*}
		F_0(t) &= 2t+Mn, \\
		F_{-,\tau}(t) &= \frac{(t-(r-\tau)n)_{n}}{n!}, \quad \tau=0,1,2,\ldots,r-1, \\
		F_{+,\tau}(t) &= \frac{(t+(M+\tau)n+1)_{n}}{n!}, \quad \tau=0,1,2,\ldots,r-1, \\
		G_j(t) &=  \frac{((M-2\delta_j)n)!}{(t+\delta_jn)_{(M-2\delta_j)n+1}}, \quad j=1,2,\ldots,J.
	\end{align*}
	From the definition of the rational function $R_n(t)$ (see Equation \eqref{defi_R_n(t)}), we have 
	\begin{equation}
		\label{product}
		R_n(t)(t+k)^s = F_{0}(t) \cdot \prod_{\tau=0}^{r-1} F_{-,\tau}(t)F_{+,\tau}(t) \cdot \prod_{j=1}^{J} \left(G_j(t)(t+k)\right)^{s/J}. 
	\end{equation}
	By Lemma \ref{lem_F}, for any polynomial $F(t)$ of the form $F_0(t)$, $F_{-,\tau}(t)$, $F_{+,\tau}(t)$, we have
	\begin{equation}
		\label{F_is_good}
		D_{n}^{\lambda} \cdot \frac{1}{\lambda!} F^{(\lambda)}(t) \big|_{t=-k} \in \mathbb{Z} 
	\end{equation} 
	for any $\lambda \in \mathbb{Z}_{\geqslant 0}$.
	By \eqref{lem_G_1} of Lemma \ref{lem_G} (with $a_0 = \delta_1n$ and $b_0 = (M-\delta_1)n$), we have
	\begin{equation}
		\label{G_is_good}
		D_{(M-2\delta_1)n}^{\lambda} \cdot \frac{1}{\lambda!} \left( G_j(t)(t+k) \right)^{(\lambda)} \big|_{t=-k} \in \mathbb{Z} 
	\end{equation}
	for any $j=1,2,\ldots,J$ and any $\lambda \in \mathbb{Z}_{\geqslant 0}$. 
	Now, substituting \eqref{product} into \eqref{a_i_k_determined_by_R_n}, applying the Leibniz rule, and using \eqref{F_is_good} and \eqref{G_is_good}, we obtain
	\begin{equation}
		\label{arith_anik_1}
		D_{(M-2\delta_1)n}^{s-i} \cdot a_{n,i,k} \in \mathbb{Z}.
	\end{equation} 
	
	Moreover, for any prime $q$ such that $\sqrt{Mn} < q \leqslant (M-2\delta_1)n$, taking also \eqref{lem_G_2} of Lemma \ref{lem_G} into consideration, and noting that $q>\sqrt{Mn}>s$ (because $n > s^2$), we have
	\begin{align}
		v_q\left( a_{n,i,k} \right) &= v_q\left( \left( R_n(t)(t+k)^{s} \right)^{(s-i)} \Big|_{t=-k} \right) \notag\\
		&\geqslant -(s-i) + \frac{s}{J}\cdot\sum_{j=1}^{J} \left( \left\lfloor \frac{(M-2\delta_j)n}{q}  \right\rfloor - \left\lfloor \frac{k-\delta_jn}{q}  \right\rfloor - \left\lfloor \frac{(M-\delta_j)n - k}{q}  \right\rfloor \right)  \notag\\
		&\geqslant -(s-i) + \frac{s}{J}\cdot\phi\left( \frac{n}{q} \right), \label{arith_anik_2}
	\end{align}
	where the function $\phi(\cdot)$ is given by \eqref{defi_phi}. 
	Combining \eqref{arith_anik_1} and \eqref{arith_anik_2}, we obtain the desired conclusion
	\[ 
	\Phi_n^{-s/J} D_{(M-2\delta_1)n}^{s-i} \cdot  a_{n,i,k} \in \mathbb{Z}. \qedhere
	\]
\end{proof}

The arithmetic property of $a_{n,i,k}$ propagates to the arithmetic property of $\rho_{n,i}$ (defined by \eqref{defi_rho_i} and \eqref{defi_rho_0}), as the following lemma shows.

\begin{lemma}
	\label{lem_arith_rhoi}
	For any integer $n > s^2$, we have
	\[ 
	\Phi_n^{-s/J} D_{(M-2\delta_1)n}^{s-i} \cdot \rho_{n,i} \in \mathbb{Z} \quad (i=1,2,\ldots,s), 
	\]
	and
	\[ 
	\Phi_n^{-s/J} \prod_{j=1}^{J} D_{M_jn}^{s/J} \cdot \rho_{n,0} \in \mathbb{Z}, 
	\]
	where 
	\[ 
	M_j = \max\{ M-2\delta_1,~M-\delta_j \}. 
	\]
\end{lemma}

\begin{proof}
	The first assertion follows directly from \eqref{defi_rho_i} and Lemma \ref{lem_arith_anik}. 
	
	Now, using \eqref{defi_rho_0}, we have
	\begin{equation}
		\label{lem4.4_1}
		\rho_{n,0} = -\sum_{i=1}^{s}\sum_{k=\delta_1n}^{(M-\delta_1)n} a_{n,i,k} \sum_{\ell=1}^{k} \frac{1}{\ell^i}. 
	\end{equation}
	Fix any $i \in \{1,2,\ldots,s\}$ and let 
	\[ 
	j^{*} = \left\lceil  \frac{iJ}{s} \right\rceil. 
	\] 
	Note that we have $a_{n,i,k} = 0$ unless $ \delta_{j^{*}}n \leqslant k \leqslant (M-\delta_{j^{*}})n$ (by examining the order of the pole at $t=-k$ in \eqref{defi_R_n(t)}). 
	Thus,
	\begin{equation}
		\label{lem4.4_2}
		\sum_{k=\delta_1n}^{(M-\delta_1)n} a_{n,i,k} \sum_{\ell=1}^{k} \frac{1}{\ell^i} = \sum_{k=\delta_{j^{*}}n}^{(M-\delta_{j^{*}})n} a_{n,i,k} \sum_{\ell=1}^{k} \frac{1}{\ell^i}. 
	\end{equation}
	Clearly,  
	\[ 
	D_{(M-\delta_{j^{*}})n}^{i} \cdot  \sum_{\ell=1}^{k} \frac{1}{\ell^{i}} \in \mathbb{Z} \quad\text{for any}\ k \leqslant (M-\delta_{j^{*}})n.
	\] 
	On the other hand, Lemma \ref{lem_arith_anik} implies
	\[ 
	\Phi_n^{-s/J} D_{(M-2\delta_1)n}^{s-i} \cdot a_{n,i,k} \in \mathbb{Z}. 
	\]
	Noting that $D_{(M-\delta_{j^{*}})n}^{i}  D_{(M-2\delta_1)n}^{s-i}$ is a divisor of $\prod_{j=1}^{J} D_{M_jn}^{s/J}$, we obtain
	\begin{equation}
		\label{lem4.4_3}
		\Phi_n^{-s/J} \prod_{j=1}^{J} D_{M_jn}^{s/J} \cdot \sum_{k=\delta_{j^{*}}n}^{(M-\delta_{j^{*}})n} a_{n,i,k} \sum_{\ell=1}^{k} \frac{1}{\ell^i} \in \mathbb{Z}. 
	\end{equation}
	Combining \eqref{lem4.4_1}, \eqref{lem4.4_2}, and \eqref{lem4.4_3}, we conclude that
	\[ 
	\Phi_n^{-s/J} \prod_{j=1}^{J} D_{M_jn}^{s/J} \cdot \rho_{n,0} \in \mathbb{Z}. \qedhere
	\]
\end{proof}

\section{Asymptotic estimates for the classical case}
\label{Sect_asym}

In this section, we will study the asymptotic behavior of the linear form $S_n$ (defined by \eqref{defi_Sn}), the coefficients $\rho_{n,i}$ (defined by \eqref{defi_rho_i} and \eqref{defi_rho_0}), and the factor $\Phi_n$ (defined by \eqref{defi_Phi}) as $n \to \infty$. 

\begin{lemma}
	\label{lem_ana_S_n}
	We have 
	\begin{equation}
		\label{S_n_estimate}
		S_n = \exp\left( -\alpha(s)n + o(n) \right) \quad \text{as~} n \to \infty, 
	\end{equation}
	where $\alpha(s)$ is a constant independent of $n$. 
	Moreover, we have
	\begin{equation}
		\label{alpha}
		\alpha(s) \sim \frac{s \log s}{J} \sum_{j=1}^{J} (M-2\delta_j) \quad \text{as~} s \to \infty. 
	\end{equation}
\end{lemma}

\begin{proof}
	Consider the Lebesgue integral
	\begin{equation}
		\label{defi_I_n}
		I_n = \int_{[0,1]^{s+1}} F(x_0,x_1,\ldots,x_s)^n G(x_0,x_1,\ldots,x_s)\,\mathrm{d}x_0\mathrm{d}x_1 \cdots \mathrm{d}x_s, 
	\end{equation} 
	where
	\begin{align}
		&~ F(x_0,x_1,\ldots,x_s) \notag\\
		=&~  \frac{x_0^{r}(1-x_0)^{M}\prod_{j=1}^{J}\prod_{k=1}^{s/J} x_{(j-1)s/J +k}^{r+\delta_j} (1-x_{(j-1)s/J +k})^{M-2\delta_j} }{(1-x_0x_1\cdots x_s)^{2r+M}} \label{F}
	\end{align}
	for $(x_0,x_1,\ldots,x_s) \in [0,1]^{s+1} \setminus \{(1,1,\ldots,1)\}$ with $F(1,1,\ldots,1)=0$, and
	\begin{equation*}
		G(x_0,x_1,\ldots,x_s) = \frac{1+x_0x_1\cdots x_{s}}{(1-x_0x_1\cdots x_s)^3}
	\end{equation*}
	for $(x_0,x_1,\ldots,x_s) \in [0,1]^{s+1} \setminus \{(1,1,\ldots,1)\}$ with $G(1,1,\ldots,1)=0$.
	
	We claim that 
	\begin{equation}
		\label{S=I}
		S_n = \frac{((2r+M)n+2)!}{(Mn)! \cdot n!^{2r}} I_n.
	\end{equation}
	In fact, using the expansion 
	\begin{align*}
		&\frac{1+x_0x_1\ldots x_s}{(1-x_0x_1\cdots x_s)^{(2r+M)n+3}} \\
		=& \frac{1}{((2r+M)n+2)!}\sum_{\nu=1}^{\infty} (\nu)_{(2r+M)n+1} \cdot ((2r+M)n+2\nu)  \cdot x_0^{\nu-1}x_1^{\nu-1}\cdots x_{s}^{\nu-1} 
	\end{align*}
	and the well-known identity for Euler's Beta function
	\[ 
	\int_{0}^{1} x^{(r+\delta)n+\nu-1}(1-x)^{(M-2\delta)n}~\mathrm{d}x = \frac{((M-2\delta)n)!}{((r+\delta)n+\nu)_{(M-2\delta)n+1}}, 
	\]
	we deduce from Beppo Levi's lemma and Tonelli's theorem that
	\begin{align*}
		I_n =&~ \frac{(Mn)!}{((2r+M)n+2)!} \\ &\times \sum_{\nu=1}^{\infty} \frac{(\nu)_{(2r+M)n+1} \cdot ((2r+M)n+2\nu)}{(rn+\nu)_{Mn+1}} \prod_{j=1}^{J} \frac{((M-2\delta_j)n)!^{s/J}}{((r+\delta_j)n+\nu)_{(M-2\delta_j)n+1}^{s/J}} \\
		=&~ \frac{(Mn)!\cdot n!^{2r}}{((2r+M)n+2)!} \sum_{\nu=1}^{\infty} R_n(rn+\nu).
	\end{align*}
	Then, by \eqref{defi_Sn} and the fact that $R_n(\nu)=0$ for $\nu=1,2,3,\ldots,rn$, we conclude that \eqref{S=I} holds.

	It is easy to see that $G(x_0,x_1,\ldots,x_s)$ is Lebesgue integrable over $[0,1]^{s+1}$. Also, $G(x_0,x_1,\ldots,x_s)$ is positive and continuous over $[0,1]^{s+1} \setminus \{(1,1,\ldots,1)\}$. 
	We claim that $F(x_0,x_1,\ldots,x_s)$ is continuous on $[0,1]^{s+1}$. 
	In fact, the only possible discontinuity of $F(x_0,x_1,\ldots,x_s)$ is at $(1,1,\ldots,1)$. 
	Since $1-x_0x_1\ldots x_s \geqslant 1-x_i$ for every $i=0,1,\ldots,s$ on $[0,1]^{s+1}$, we have
	\[ 
	(1-x_0x_1\cdots x_s)^{s+1} \geqslant \prod_{i=0}^{s} (1-x_i), 
	\]
	and hence 
	\begin{align}
		&~F(x_0,x_1,\ldots,x_s) \notag\\
		\leqslant&~ x_0^{r}(1-x_0)^{M-\frac{2r+M}{s+1}}\prod_{j=1}^{J}\prod_{k=1}^{s/J} x_{(j-1)s/J +k}^{r+\delta_j} (1-x_{(j-1)s/J +k})^{M-2\delta_j- \frac{2r+M}{s+1}} \label{F2}
	\end{align}
	for all $(x_0,x_1,\ldots,x_s) \in [0,1]^{s+1} \setminus \{ (1,1,\ldots,1) \}$. 
	Since $r = \lfloor s/\log^2 s \rfloor$ and $s$ is sufficiently large (so that $(2r+M)/(s+1)<1$), inequality \eqref{F2} implies the continuity of $F(x_0,x_1,\ldots,x_s)$ at $(1,1,\ldots,1)$. 
	Thus, by \eqref{defi_I_n}, we have
	\[ 
	\lim_{n \to \infty} I_n^{1/n} = \max_{(x_0,x_1,\ldots,x_s) \in [0,1]^{s+1}} F(x_0,x_1,\ldots,x_s). 
	\]
	Then, by \eqref{S=I} and Stirling's formula, we have
	\[ 
	S_n = \exp\left( -\alpha(s)n + o(n) \right) \quad\text{as~} n \to \infty, 
	\]
	where
	\begin{align}
		\alpha(s) =&~  M\log M - (2r+M)\log(2r+M) \notag\\
		&- \log \max_{(x_0,x_1,\ldots,x_s) \in [0,1]^{s+1}} F(x_0,x_1,\ldots,x_s). \label{defi_alpha}
	\end{align}
	
	Finally, by \eqref{defi_alpha} and $r =\lfloor s/\log^2 s \rfloor$, to prove the last assertion about the asymptotic behavior of $\alpha(s)$ as $s \to \infty$, it remains to prove that
	\begin{equation}\label{alpha_estimate}
		-\log \max_{(x_0,x_1,\ldots,x_s) \in [0,1]^{s+1}} F(x_0,x_1,\ldots,x_s) \sim \frac{s\log s}{J}\sum_{j=1}^{J} (M-2\delta_j) \quad \text{as~} s\to \infty.
	\end{equation}
	In fact, taking $x_0=x_1=\cdots=x_s = 1-1/r$ in \eqref{F}, we have (recall $r=\lfloor s/\log^2 s\rfloor$)
	\begin{align}
		&\log \max_{(x_0,x_1,\ldots,x_s) \in [0,1]^{s+1}} F(x_0,x_1,\ldots,x_s) \notag\\
		\geqslant& \left( (s+1)r + \frac{s}{J}\sum_{j=1}^{J} \delta_j \right)\log\left( 1-\frac{1}{r} \right) - \left( M + \frac{s}{J}\sum_{j=1}^{J} (M-2\delta_j) \right) \log r \notag\\
		&- (2r+M)\log\left( 1-\left( 1-\frac{1}{r} \right)^{s+1} \right) \notag\\
		=&~ -\frac{s\log s}{J}\sum_{j=1}^{J} (M-2\delta_j) - O(s\log\log s) \quad \text{as~} s \to \infty. \label{F3}
	\end{align}
	On the other hand, applying the elementary inequality 
	\[ 
	x^{a}(1-x)^b \leqslant \frac{a^ab^b}{(a+b)^{a+b}} \quad(\text{for any}\  a>0,\ b>0,\ \text{and}\ x \in [0,1])
	\] 
	to \eqref{F2}, we have
	\begin{align}
		&\log \max_{(x_0,x_1,\ldots,x_s) \in [0,1]^{s+1}} F(x_0,x_1,\ldots,x_s) \notag\\
		\leqslant&~ r\log r  + \left( M-\frac{2r+M}{s+1} \right)\log\left( M-\frac{2r+M}{s+1} \right) \notag\\
		&-\left( r+M - \frac{2r+M}{s+1} \right)\log\left(  r+M - \frac{2r+M}{s+1} \right) \notag\\ &+\frac{s}{J}\sum_{j=1}^{J}(r+\delta_j)\log(r+\delta_j)   \notag\\
		&+ \frac{s}{J}\sum_{j=1}^{J} \left( M-2\delta_j - \frac{2r+M}{s+1} \right) \log\left( M-2\delta_j - \frac{2r+M}{s+1} \right) \notag\\
		&- \frac{s}{J}\sum_{j=1}^{J} \left( r+M-\delta_j - \frac{2r+M}{s+1} \right)\log\left(  r+M-\delta_j - \frac{2r+M}{s+1} \right) \notag\\
		=&~ \frac{s}{J}\sum_{j=1}^{J} (r+\delta_j)\log r - \frac{s}{J}\sum_{j=1}^{J}\left( r+M-\delta_j \right)\log r + O(s) \notag\\
		=&~ -\frac{s\log s}{J}\sum_{j=1}^{J} (M-2\delta_j) + O(s\log\log s) \quad \text{as~} s \to \infty. \label{F4}
	\end{align}
	Combining \eqref{F3} and \eqref{F4}, we conclude that \eqref{alpha_estimate} holds. 
	The proof of Lemma \ref{lem_ana_S_n} is complete.
\end{proof}

Recall that $\rho_{n,i}$ ($i=0,1,2,\ldots,s$) are defined by \eqref{defi_rho_0} and \eqref{defi_rho_i}. 

\begin{lemma}
	\label{lem_ana_rho}
	We have 
	\begin{equation}
		\label{rho_estimate}
		\max_{0 \leqslant i \leqslant s} |\rho_{n,i}| \leqslant \exp\left( \beta(s)n + o(n) \right) \quad \text{as~} n \to \infty, 
	\end{equation}
	where $\beta(s)$ is a constant independent of $n$. 
	Moreover, we have
	\begin{equation}
		\label{beta}
		\beta(s) \sim \log2 \cdot \frac{s}{J} \sum_{j=1}^{J} (M - 2\delta_j) \quad\text{as~} s \to \infty. 
	\end{equation}
\end{lemma}

\begin{proof}
	By \eqref{defi_rho_0} and \eqref{defi_rho_i}, we have
	\begin{equation}
		\label{rho<a}
		\max_{0 \leqslant i \leqslant s} |\rho_{n,i}| \leqslant (Mn+1)^{2}s \cdot \max_{i,k} |a_{n,i,k}|, 
	\end{equation}
	where $\max_{i,k}$ is taken over all $i \in \{ 1,2,\ldots,s \}$ and $k \in [\delta_1n, (M-\delta_1)n] \cap \mathbb{Z}$. 
	
	Now fix any pair $(i,k)$ in the range above. 
	By \eqref{defi_anik} and Cauchy's integral formula, we have
	\[ 
	a_{n,i,k} = \frac{1}{2\pi \sqrt{-1}} \int_{|z+k|=1/10} (z+k)^{i-1}R_n(z)\,\mathrm{d}z. 
	\]
	Therefore, we have (recall \eqref{defi_R_n(t)})
	\begin{align}
		|a_{n,i,k}| &\leqslant \max_{|z+k| = 1/10} |R_n(z)| \notag\\
		&= \frac{\prod_{j=1}^{J} ((M-2\delta_j)n)!^{s/J}}{n!^{2r}}  \notag\\
		&\quad  \times \max_{|z+k| = 1/10} |2z+Mn|  \frac{|(z-rn)_{rn} (z+Mn+1)_{rn}|}{ \prod_{j=1}^{J} |(z+\delta_jn)_{(M-2\delta_j)n+1}|^{s/J}}. \label{|a|<maxR}
	\end{align}
	In the following, the complex number $z$ is always assumed to be on the circle $|z+k|=1/10$. 
	We will use the triangle inequality to bound \eqref{|a|<maxR} from above. 
	
	Obviously, we have
	\begin{equation}
		\label{est1}
		\max_{|z+k| = 1/10} |2z+Mn| \leqslant 4Mn.
	\end{equation}
	For any $\nu \in \{1,2,\ldots,rn\}$, we have
	\[ 
	|(z-\nu)(z+Mn+\nu)| \leqslant (k+\nu+1)(Mn-k+\nu+1) \leqslant \left( \frac{Mn}{2} + \nu + 1 \right)^2. 
	\]
	Therefore, 
	\begin{align}
		&~\max_{|z+k| = 1/10} |(z-rn)_{rn} (z+Mn+1)_{rn}| \notag\\
		=&~ \max_{|z+k| = 1/10} \prod_{\nu=1}^{rn} |(z-\nu)(z+Mn+\nu)| \notag\\
		\leqslant&~ \prod_{\nu=1}^{rn} \left( \left\lceil\frac{Mn}{2}\right\rceil + \nu + 1 \right)^2 \notag\\
		=&~ \frac{\left(\lceil \frac{Mn}{2} \rceil + rn +1\right)!^2}{\left( \lceil\frac{Mn}{2}\rceil + 1 \right)!^2}. \label{est2}
	\end{align}
	Now, if $\delta_jn < k < (M-\delta_j)n$, then
	\begin{align*}
		|(z+\delta_jn)_{(M-2\delta_j)n+1}| &\geqslant (k-\delta_jn-1)! \cdot \frac{9}{10} \cdot \frac{1}{10} \cdot \frac{9}{10} \cdot ((M-\delta_j)n-k-1)! \\
		&\geqslant \frac{1}{10^3M^2n^2} \cdot (k-\delta_jn)! \cdot ((M-\delta_j)n-k)!,
	\end{align*}
	and hence
	\[ 
	\frac{((M-2\delta_j)n)!}{|(z+\delta_jn)_{(M-2\delta_j)n+1}|} \leqslant 10^3M^2n^2 \binom{(M-2\delta_j)n}{k-\delta_jn} \leqslant 10^3M^2n^2 \cdot 2^{(M-2\delta_j)n}. 
	\]
	If $k \leqslant \delta_j n$, then 
	\begin{align*}
		|(z+\delta_jn)_{(M-2\delta_j)n+1}| &\geqslant \frac{1}{10} \cdot \frac{9}{10} \cdot ((M-2\delta_j)n-1)! \\
		&\geqslant \frac{1}{10^2Mn} ((M-2\delta_j)n)!,
	\end{align*}
	and hence
	\[ 
	\frac{((M-2\delta_j)n)!}{|(z+\delta_jn)_{(M-2\delta_j)n+1}|} \leqslant 10^2Mn. 
	\]
	Similarly, if $k \geqslant (M-\delta_j)n$, we also have
	\[ 
	\frac{((M-2\delta_j)n)!}{|(z+\delta_jn)_{(M-2\delta_j)n+1}|} \leqslant 10^2Mn. 
	\]
	In any case, we always have
	\begin{equation}
		\label{est3}
		\max_{|z+k| = 1/10} \frac{((M-2\delta_j)n)!}{|(z+\delta_jn)_{(M-2\delta_j)n+1}|} \leqslant 10^3M^2n^2 \cdot 2^{(M-2\delta_j)n}.
	\end{equation}
	Combining inequalities \eqref{rho<a}--\eqref{est3} and using Stirling's formula, we have
	\[ 
	\max_{0 \leqslant i \leqslant s} |\rho_{n,i}| \leqslant \exp\left( \beta(s)n + o(n) \right) \quad \text{as~} n \to \infty,  
	\]
	where
	\[ 
	\beta(s) = \left(2r + M \right)\log\left(r+\frac{M}{2} \right) - M\log\frac{M}{2} + \log2 \cdot \frac{s}{J} \sum_{j=1}^{J} (M-2\delta_j). 
	\]
	Since $r = \lfloor s/\log^2 s \rfloor$, we have
	\[ 
	\beta(s) \sim \log2 \cdot \frac{s}{J} \sum_{j=1}^{J} (M-2\delta_j) \quad \text{as~} s \to \infty. \qedhere
	\]
\end{proof}

At the end of this section, we estimate the factor $\Phi_n$ (defined by \eqref{defi_Phi}). 
It is well known that the prime number theorem implies 
\begin{equation}
	\label{PNT}
	D_m =e^{m+o(m)} \quad \text{as~} m \to \infty. 
\end{equation}
The following lemma is also a corollary of the prime number theorem. 
\begin{lemma}[{cf. \cite[Lemma 4.4]{Zud2002}}]
	\label{Phi_est}
	For the factor $\Phi_n$ defined by \eqref{defi_Phi}, we have
	\[ 
	\Phi_n = \exp\left( \varpi n + o(n) \right) \quad \text{as~} n \to \infty,  
	\]
	where the constant $\varpi$ is given by
	\begin{equation}
		\label{defi_varpi}
		\varpi = \int_{0}^{1} \phi(x)\,\mathrm{d}\psi(x) - \int_{0}^{1/(M-2\delta_1)} \phi(x)\,\frac{\mathrm{d}x}{x^2}.  
	\end{equation}
	Here, the function $\psi(x) = \Gamma^{\prime}(x)/\Gamma(x)$ is the digamma function, and the function $\phi(x)$ is defined by \eqref{defi_phi}.
\end{lemma}

\section{Proof of Theorem \ref{mainthm}}\label{Sect_proo}
In this section, we will first reduce the proof of Theorem \ref{mainthm} to a computational task. 
Then, we will prove Theorem \ref{mainthm} using some very simple parameters. 

\begin{proposition}\label{prop}
	Let $\psi(\cdot)$ be the digamma function. 
	For any collection of nonnegative integers $(M,\delta_1,\delta_2,\ldots,\delta_J)$ under the constraints $J \geqslant 1$ and
	\[ 	
	0 \leqslant \delta_1 \leqslant \delta_2 \leqslant \cdots \leqslant \delta_J < \frac{M}{2}, 
	\]
	we define
	\[ 
	\varpi = \int_{0}^{1} \phi(x)\,\mathrm{d}\psi(x) - \int_{0}^{1/(M-2\delta_1)} \phi(x)\,\frac{\mathrm{d}x}{x^2},   
	\]
	where the function $\phi(\cdot)$ is given by
	\[ 
	\phi(x) = \inf_{y \in \mathbb{R}} \sum_{j=1}^{J}  \left( \lfloor(M-2\delta_j)x\rfloor - \lfloor y-\delta_j x \rfloor - \lfloor (M-\delta_j)x - y \rfloor \right). 
	\]
	Then, we have
	\[ 
	\dim_{\mathbb{Q}}\operatorname{Span}_{\mathbb{Q}}\left( 1,\zeta(3),\zeta(5),\ldots,\zeta(s-1) \right) \geqslant (C-o(1)) \cdot \log s
	\]
	as the positive even integer $s \to \infty$, where the constant $C$ is defined by
	\begin{equation}\label{defi_C}
		C = \frac{\sum_{j=1}^{J} (M-2\delta_j) }{\log2 \cdot \sum_{j=1}^{J}(M-2\delta_j)  - \varpi  + \sum_{j=1}^{J} \max\{ M-2\delta_1,~M-\delta_j \}}.
	\end{equation}  
\end{proposition}

\begin{proof}
	For any integer $n>s^2$, consider  
	\[ 
	\widehat{S}_n := \Phi_n^{-s/J} \prod_{j=1}^{J} D_{M_jn}^{s/J} \cdot S_n, 
	\]
	where $M_j = \max\{ M-2\delta_1,~M-\delta_j \}$. 
	By Lemmas \ref{lem_linearforms} and \ref{lem_arith_rhoi}, 
	\[ 
	\widehat{S}_n = \widehat{\rho}_{n,0} + \sum_{3 \leqslant i \leqslant s-1 \atop i \text{~odd}} \widehat{\rho}_{n,i}\zeta(i)  
	\]
	is a linear form in $1$ and odd zeta values with integer coefficients: 
	\[ 
	\widehat{\rho}_{n,i} := \Phi_n^{-s/J} \prod_{j=1}^{J} D_{M_jn}^{s/J} \cdot \rho_{n,i} \in \mathbb{Z}, \quad i=0,3,5,\ldots,s-1. 
	\]
	By Lemma \ref{lem_ana_S_n}, Equation \eqref{PNT}, and Lemma \ref{Phi_est}, we have
	\[ 
	\widehat{S}_n = \exp\left( -\widehat{\alpha}(s)n+o(n) \right)  \quad \text{as~} n\to \infty, 
	\]
	where
	\[ 
	\widehat{\alpha}(s) = \alpha(s) + \frac{s}{J}\varpi - \frac{s}{J}\sum_{j=1}^{J} \max\left\{ M-2\delta_1,~M-\delta_j \right\}. 
	\]
	By Lemma \ref{lem_ana_rho}, Equation \eqref{PNT}, and Lemma \ref{Phi_est}, we have
	\[ 
	\max_{i=0,3,5,\ldots,s-1} |\widehat{\rho}_{n,i}| \leqslant \exp\left( \widehat{\beta}(s)n + o(n) \right) \quad \text{as~} n\to \infty, 
	\]
	where
	\[ 
	\widehat{\beta}(s) = \beta(s) - \frac{s}{J}\varpi + \frac{s}{J}\sum_{j=1}^{J} \max\left\{ M-2\delta_1,~M-\delta_j \right\}. 
	\]
	
	By \eqref{alpha} and \eqref{beta}, we have the following asymptotic estimates for $\widehat{\alpha}(s)$ and $\widehat{\beta}(s)$ as $s \to \infty$:
	\begin{align}
		\widehat{\alpha}(s) &\sim  \frac{s\log s}{J}\sum_{j=1}^{J}(M-2\delta_j),   \label{alphahat}\\
		\widehat{\beta}(s) &\sim  \log 2 \cdot \frac{s}{J}\sum_{j=1}^{J}(M-2\delta_j) - \frac{s}{J}\varpi + \frac{s}{J}\sum_{j=1}^{J} \max\left\{ M-2\delta_1,~M-\delta_j \right\}. \label{betahat}
	\end{align}
	Applying Nesterenko's linear independence criterion (Theorem \ref{Nes}) to the sequence of linear forms $\{ \widehat{S}_n \}_{n > s^2}$, and using Equations \eqref{alphahat} and \eqref{betahat}, we obtain
	\begin{align}
		\dim_{\mathbb{Q}}\operatorname{Span}_{\mathbb{Q}}\left( 1,\zeta(3),\zeta(5),\ldots,\zeta(s-1) \right) &\geqslant 1 +\frac{\widehat{\alpha}(s)}{\widehat{\beta}(s)} \notag \\
		&= (C-o(1)) \cdot \log s \quad \text{as~} s \to \infty, \label{dim>Clogs}
	\end{align}
	where the constant $C$ is given by
	\begin{equation*}
		C = \frac{\sum_{j=1}^{J} (M-2\delta_j) }{\log2 \cdot \sum_{j=1}^{J}(M-2\delta_j)  - \varpi  + \sum_{j=1}^{J} \max\{ M-2\delta_1,~M-\delta_j \}}. 
	\end{equation*}
	
	We have restricted $s$ to $2J\mathbb{Z}_{>0}$ in our constructions of rational functions and linear forms. 
	Therefore, \eqref{dim>Clogs} holds as $s \in 2J\mathbb{Z}_{>0}$ and $s \to \infty$. 
	However, if we replace the condition `$s \in 2J\mathbb{Z}_{>0}$' by `$s \in 2\mathbb{Z}_{>0}$', inequality \eqref{dim>Clogs} remains valid.
	Thus, the proof of Proposition \ref{prop} is complete.
\end{proof}

If we take $M=1$, $\delta_1=0$ (with $J=1$) in Proposition \ref{prop}, then we find $\phi(x) \equiv 0$, $\varpi = 0$, and $C = 1/(1+\log 2)$. 
Thus, we recover the Ball--Rivoal theorem \cite{BR2001}. 
A simple choice of parameters to improve the Ball--Rivoal theorem is: $M=6$, $\delta_1=0$, $\delta_2=1$ (with $J=2$). 
This collection of parameters leads to the proof of Theorem \ref{mainthm}:

\begin{proof}[Proof of Theorem \textup{\ref{mainthm}}]
	Taking $M=6$, $\delta_1=0$, $\delta_2 = 1$ (with $J=2$) in Proposition \ref{prop}, the function $\phi(\cdot)$ defined by \eqref{defi_phi} becomes
	\begin{align*}
		\phi(x) &= \inf_{y \in \mathbb{R}} \left(  \lfloor6x\rfloor - \lfloor y \rfloor - \lfloor 6x - y \rfloor + \lfloor4x\rfloor - \lfloor y- x \rfloor - \lfloor 5x - y \rfloor \right).
	\end{align*} 
	Note that the two-variable function
	\[ 
	\phi(x,y) := \lfloor6x\rfloor - \lfloor y \rfloor - \lfloor 6x - y \rfloor + \lfloor4x\rfloor - \lfloor y- x \rfloor - \lfloor 5x - y \rfloor 
	\]
	satisfies $\phi(x,y)=\phi(\{x\},\{y\})$. 
	In the $xy$-plane, the lines of the form $6x=k$, $y=k$, $6x-y=k$, $4x=k$, $y-x=k$, and $5x-y=k$ (with $k=0,\pm1,\pm2,\ldots$) divide the unit square $[0,1]^2$ into several polygons. 
	The function $\phi(x,y)$ is constant in the interior of each polygon. 
	Figure \ref{figure} illustrates the constant value of $\phi(x,y)$ in the interior of each polygon.
	\begin{figure}[h]
		\centering
		\begin{tikzpicture}[line cap=round,line join=round,>=triangle 45,x=1cm,y=1cm,scale=10]
			\clip(-0.1,-0.1) rectangle (1.1,1.1);
			\draw [line width=1pt] (0,0)-- (0,1);
			\draw [line width=1pt] (0,0)-- (1,0);
			\draw [line width=1pt] (1,0)-- (1,1);
			\draw [line width=1pt] (0,1)-- (1,1);
			\draw [line width=1pt] (1/6,0)-- (1/6,1);
			\draw [line width=1pt] (2/6,0)-- (2/6,1);
			\draw [line width=1pt] (3/6,0)-- (3/6,1);
			\draw [line width=1pt] (4/6,0)-- (4/6,1);
			\draw [line width=1pt] (5/6,0)-- (5/6,1);
			\draw [line width=1pt] (0,0)-- (1/6,1);
			\draw [line width=1pt] (1/6,0)-- (2/6,1);
			\draw [line width=1pt] (2/6,0)-- (3/6,1);
			\draw [line width=1pt] (3/6,0)-- (4/6,1);
			\draw [line width=1pt] (4/6,0)-- (5/6,1);
			\draw [line width=1pt] (5/6,0)-- (1,1);
			\draw [line width=1pt] (1/4,0)-- (1/4,1);
			\draw [line width=1pt] (3/4,0)-- (3/4,1);
			\draw [line width=1pt] (0,0)-- (1,1);
			\draw [line width=1pt] (0,0)-- (1/5,1);
			\draw [line width=1pt] (1/5,0)-- (2/5,1);
			\draw [line width=1pt] (2/5,0)-- (3/5,1);
			\draw [line width=1pt] (3/5,0)-- (4/5,1);
			\draw [line width=1pt] (4/5,0)-- (1,1);
			
			\draw (0,-0.06) node[anchor=south] {$0$};
			\draw (1/6,-0.07) node[anchor=south] {$\frac{1}{6}$};
			\draw (1/5,-0.07) node[anchor=south] {$\frac{1}{5}$};
			\draw (1/4,-0.07) node[anchor=south] {$\frac{1}{4}$};
			\draw (1/3,-0.07) node[anchor=south] {$\frac{1}{3}$};
			\draw (2/5,-0.07) node[anchor=south] {$\frac{2}{5}$};
			\draw (1/2,-0.07) node[anchor=south] {$\frac{1}{2}$};
			\draw (3/5,-0.07) node[anchor=south] {$\frac{3}{5}$};
			\draw (2/3,-0.07) node[anchor=south] {$\frac{2}{3}$};
			\draw (3/4,-0.07) node[anchor=south] {$\frac{3}{4}$};
			\draw (4/5,-0.07) node[anchor=south] {$\frac{4}{5}$};
			\draw (5/6,-0.07) node[anchor=south] {$\frac{5}{6}$};
			\draw (1,-0.06) node[anchor=south] {$1$};
			
			\draw (0,1) node[anchor=south] {$0$};
			\draw (1/6,1) node[anchor=south] {$\frac{1}{6}$};
			\draw (1/5,1) node[anchor=south] {$\frac{1}{5}$};
			\draw (1/4,1) node[anchor=south] {$\frac{1}{4}$};
			\draw (1/3,1) node[anchor=south] {$\frac{1}{3}$};
			\draw (2/5,1) node[anchor=south] {$\frac{2}{5}$};
			\draw (1/2,1) node[anchor=south] {$\frac{1}{2}$};
			\draw (3/5,1) node[anchor=south] {$\frac{3}{5}$};
			\draw (2/3,1) node[anchor=south] {$\frac{2}{3}$};
			\draw (3/4,1) node[anchor=south] {$\frac{3}{4}$};
			\draw (4/5,1) node[anchor=south] {$\frac{4}{5}$};
			\draw (5/6,1) node[anchor=south] {$\frac{5}{6}$};
			\draw (1,1) node[anchor=south] {$1$};
			
			\draw (1/12,3/4) node[anchor=south] {$2$};
			\draw (0.177,5/6) node[anchor=east] {$1$};
			\draw (0.18,0.93) node[anchor=south] {$2$};
			\draw (0.12,0.3) node[anchor=south] {$0$};
			\draw (0.1,0.02) node[anchor=south] {$1$};
			\draw (0.21,0.55) node[anchor=south] {$1$};
			\draw (0.28,0.88) node[anchor=south] {$2$};
			\draw (0.23,0.24) node[anchor=south] {$0$};
			\draw (0.158,0.16) node[anchor=west] {$2$};
			\draw (0.21,0.1) node[anchor=south] {$1$};
			\draw (0.235,0.05) node[anchor=south] {$0$};
			\draw (0.3,0.6) node[anchor=south] {$1$};
			\draw (0.305,0.35) node[anchor=south] {$0$};
			\draw (0.3,0.13) node[anchor=south] {$1$};
			\draw (0.36,0.88) node[anchor=south] {$2$};
			\draw (0.4,0.7) node[anchor=south] {$1$};
			\draw (0.355,0.26) node[anchor=south] {$2$};
			\draw (0.41,0.2) node[anchor=south] {$1$};
			\draw (0.46,0.52) node[anchor=south] {$0$};
			\draw (0.46,0.1) node[anchor=south] {$0$};
			\draw (0.535,0.85) node[anchor=south] {$2$};
			\draw (0.59,0.75) node[anchor=south] {$1$};
			\draw (0.535,0.35) node[anchor=south] {$2$};
			\draw (0.6,0.3) node[anchor=south] {$1$};
			\draw (0.64,0.68) node[anchor=south] {$0$};
			\draw (0.64,0.08) node[anchor=south] {$0$};
			\draw (0.71,0.82) node[anchor=south] {$1$};
			\draw (0.69,0.61) node[anchor=south] {$2$};
			\draw (0.71,0.4) node[anchor=south] {$1$};
			\draw (0.723,0.145) node[anchor=south] {$0$};
			\draw (0.767,0.9) node[anchor=south] {$2$};
			\draw (0.793,0.85) node[anchor=south] {$1$};
			\draw (0.82,0.818) node[anchor=south] {$0$};
			\draw (0.77,0.7) node[anchor=south] {$2$};
			\draw (0.79,0.4) node[anchor=south] {$1$};
			\draw (0.82,0.01) node[anchor=south] {$0$};
			\draw (0.885,0.91) node[anchor=south] {$1$};
			\draw (0.88,0.69) node[anchor=south] {$2$};
			\draw (0.824,0.166) node[anchor=west] {$1$};
			\draw (0.93,0.2) node[anchor=south] {$0$};
		\end{tikzpicture} 
		\caption{The values of $\phi(x,y)$.}
		\label{figure}
	\end{figure}
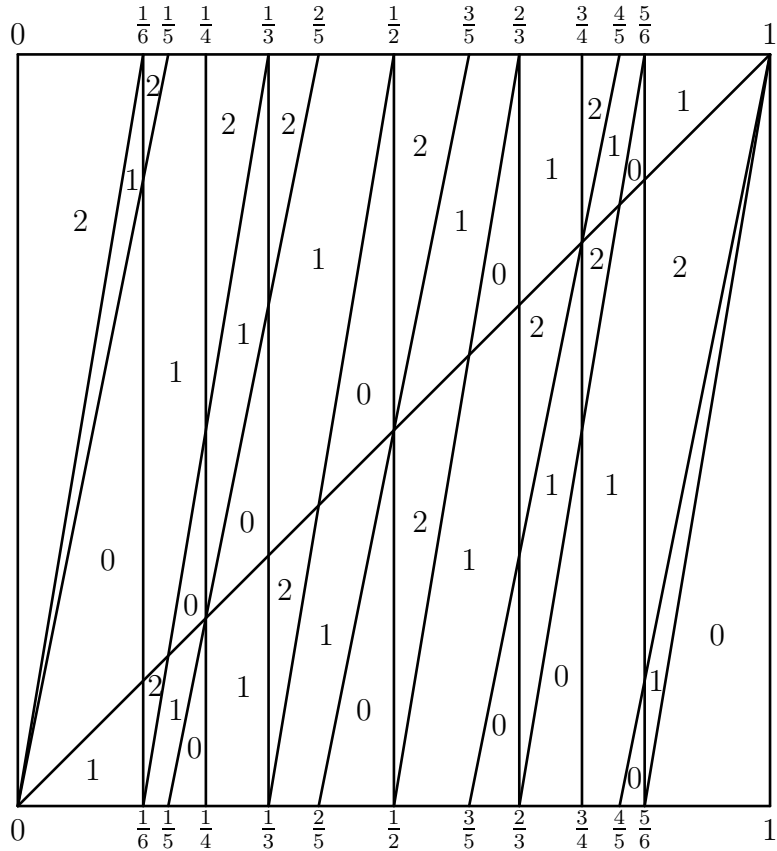
	
	Based on Figure \ref{figure}, we can express $\phi(x)$ explicitly as follows:
	\begin{align*}
		\phi(x) &= \inf_{y \in \mathbb{R}} \phi(x,y) =\min_{0 \leqslant y < 1} \phi(\{x\},y) \\
		&= \begin{cases}
			1 \quad\text{if~} \{x\} \in \left[\frac{1}{6},\frac{1}{5}\right) \cup \left[\frac{1}{3},\frac{2}{5}\right) \cup \left[\frac{1}{2},\frac{3}{5}\right) \cup \left[\frac{3}{4},\frac{4}{5}\right), \\
			0 \quad\text{otherwise.} 
		\end{cases}  
	\end{align*} 
	Therefore, by \eqref{defi_varpi}, we have
	\begin{align*}
		\varpi =&~ \psi\left(\frac{1}{5}\right) - \psi\left(\frac{1}{6}\right) + \psi\left(\frac{2}{5}\right) - \psi\left(\frac{1}{3}\right) \\
		&+ \psi\left(\frac{3}{5}\right) - \psi\left(\frac{1}{2}\right) + \psi\left(\frac{4}{5}\right) - \psi\left(\frac{3}{4}\right) \\
		=&~ 2.157479\ldots.
	\end{align*}
	Hence, by \eqref{defi_C},
	\[ 
	C = \frac{10}{10\log 2 - \varpi + 12} = \frac{1.009388\ldots}{1+\log 2}. 
	\]
	In conclusion, by Proposition \ref{prop}, we have
	\[ 
	\dim_{\mathbb{Q}}\operatorname{Span}_{\mathbb{Q}}\left( 1,\zeta(3),\zeta(5),\ldots,\zeta(s-1) \right) \geqslant \frac{1.009}{1+\log 2} \cdot \log s
	\]
	for any sufficiently large even integer $s$. 
	The proof of Theorem \ref{mainthm} is complete.
\end{proof}

\section{Preliminaries for $p$-adic zeta values}\label{Sect_Prel}

From now on, we fix a prime number $p$. 
In this section, we will recall basic facts about Volkenborn integrals, the first Bernoulli functional, $p$-adic Hurwitz zeta functions, and $p$-adic zeta values.

\subsection{Overconvergent power series on $\mathbb{Z}_p$}
Let $C(\mathbb{Z}_p,\mathbb{Q}_p)$ denote the set of all continuous functions $f: \mathbb{Z}_p \longrightarrow \mathbb{Q}_p$. 
We define the set
\[ 
C^{\dagger}(\mathbb{Z}_p,\mathbb{Q}_p):= \left\{ f \in C(\mathbb{Z}_p,\mathbb{Q}_p) ~\Big|~ f(t) = \sum_{k=0}^{\infty} a_k t^{k}, a_k \in \mathbb{Q}_p, \limsup_{k \to \infty} |a_k|_p^{1/k} < 1  \right\}. 
\]
In other words, $f \in C^{\dagger}(\mathbb{Z}_p,\mathbb{Q}_p)$ if and only if the function $f(t)$ coincides on $\mathbb{Z}_p$ with a power series in $\mathbb{Q}_p\llbracket t \rrbracket$ whose radius of convergence is strictly greater than $1$. 
If $f \in C^{\dagger}(\mathbb{Z}_p,\mathbb{Q}_p)$, we say that $f$ is an \emph{overconvergent power series}. 

If $D$ is a set containing $\mathbb{Z}_p$ and $g: D \longrightarrow \mathbb{Q}_p$ is a function such that the restriction $g|_{\mathbb{Z}_p} \in C^{\dagger}(\mathbb{Z}_p,\mathbb{Q}_p)$, then by a slight abuse of notation, we write $g \in C^{\dagger}(\mathbb{Z}_p,\mathbb{Q}_p)$.

\subsection{Volkenborn integrals}
A function $f: \mathbb{Z}_p \longrightarrow \mathbb{C}_p$ is said to be \emph{Volkenborn integrable} if the following limit exists in $\mathbb{C}_p$:
\[ 
\lim_{N \to \infty} \frac{1}{p^N}\sum_{k=0}^{p^{N}-1} f(k).  
\]
For a Volkenborn integrable function $f$, we define its \emph{Volkenborn integral} by 
\[ 
\int_{\mathbb{Z}_p}f(t) \,\mathrm{d}t:= \lim_{N \to \infty} \frac{1}{p^N}\sum_{k=0}^{p^{N}-1} f(k). 
\]

As usual, we define the Bernoulli numbers $B_k \in \mathbb{Q}$ ($k=0,1,2,\ldots$) by the generating series
\[ 
\frac{t}{e^t-1} =: \sum_{k=0}^{\infty} \frac{B_k}{k!}t^k. 
\]

\begin{lemma}\label{lem_Vol_t^k}
	Any overconvergent power series $f(t) = \sum_{k=0}^{\infty} a_k t^k \in C^{\dagger}(\mathbb{Z}_p,\mathbb{Q}_p)$ is Volkenborn integrable. 
	Moreover, we have
	\[ 
	\int_{\mathbb{Z}_p}f(t) \,\mathrm{d}t = \sum_{k=0}^{\infty} a_kB_k. 
	\]
\end{lemma}

\begin{proof}
	See \cite[Proposition, p. 270]{Rob2000}, where a stronger result is proved.
\end{proof}

The Volkenborn integral has the following behavior under translations:
\begin{lemma}[{\cite[Proposition 2, p. 265]{Rob2000}}]\label{lem_Vol_trans}
	Let $k\in \mathbb{Z}_{>0}$ and $f \in C^{\dagger}(\mathbb{Z}_p,\mathbb{Q}_p)$. Then, we have
	\[
	\int_{\mathbb{Z}_p} f(t+k)\,\mathrm{d}t= \int_{\mathbb{Z}_p} f(t)\,\mathrm{d}t+\sum_{\ell=0}^{k-1}f'(\ell).
	\]
\end{lemma}

\subsection{The first Bernoulli functional $\mathcal{L}_1$}

The following definition turns out to be useful:

\begin{definition}[{\cite[Definition 2.6]{LS2023+}}]\label{def_L_1}
	We define the \emph{first Bernoulli functional} 
	\[
	\mathcal{L}_1: C^{\dagger}(\mathbb{Z}_p,\mathbb{Q}_p) \longrightarrow \mathbb{Q}_p
	\] 
	by
	\[ 
	f = \sum_{k=0}^{\infty} a_k t^{k} \mapsto \mathcal{L}_1(f) := \sum_{k=0}^{\infty} \frac{a_kB_{k+1}}{k+1}. 
	\]
\end{definition}
Since $\limsup_{k \to \infty} |a_k|_p^{1/k} < 1$ and $|B_{k+1}/(k+1)|_p \leqslant p(k+1)$ (by the von Staudt--Clausen theorem), the sum $\sum_{k=0}^{\infty} a_kB_{k+1}/(k+1)$ is convergent in $\mathbb{Q}_p$.
The main property of $\mathcal{L}_1$ is the following:

\begin{lemma}[{\cite[Lemma 2.7]{LS2023+}}]\label{lem_L_1}
	For any $f\in C^{\dagger}(\mathbb{Z}_p,\mathbb{Q}_p)$ we have the formula
	\begin{equation*}
		\mathcal{L}_1(f')=\int_{\mathbb{Z}_p} f(t) \,\mathrm{d}t - f(0).
	\end{equation*}
\end{lemma}

\begin{proof}
	We have $f' \in C^{\dagger}(\mathbb{Z}_p,\mathbb{Q}_p)$ because $C^{\dagger}(\mathbb{Z}_p,\mathbb{Q}_p)$ is clearly closed under differentiation. 
	Suppose that $f(t) = \sum_{k=0}^{\infty} a_k t^k$. 
	Then, by Definition \ref{def_L_1} and Lemma \ref{lem_Vol_t^k}, we have
	\begin{equation*}
		\mathcal{L}_1(f') = \sum_{k=1}^{\infty} a_kB_k = \int_{\mathbb{Z}_p} f(t) \,\mathrm{d}t - f(0). \qedhere
	\end{equation*}
\end{proof}

\begin{remark}
	The first Bernoulli functional $\mathcal{L}_1$ can be viewed as a de-regularization of the first Bernoulli measure $\mu_{1,\alpha}$ over the space of overconvergent power series, as noted in \cite[Proposition 2.9]{LS2023+}. 
	This justifies the name of $\mathcal{L}_1$.
\end{remark}

\subsection{$p$-Adic Hurwitz zeta functions and $p$-adic zeta values}
In this subsection, we will recall some basic facts about $p$-adic Hurwitz zeta functions and $p$-adic zeta values.

Set $q_p = p$ if $p > 2$ and $q_2 = 4$. 
The units $\mathbb{Z}_p^\times$ of the $p$-adic integers decompose canonically as follows:
\[
\mathbb{Z}_p^\times \cong \mu_{\varphi(q_p)}(\mathbb{Z}_p)\times (1+q_p\mathbb{Z}_p).
\]
Here, $\mu_{\varphi(q_p)}(\mathbb{Z}_p)$ denotes the group of $\varphi(q_p)$-th roots of unity in $\mathbb{Z}_p$. 
The canonical projection
\[
\omega \colon \mathbb{Z}_p^\times \longrightarrow \mu_{\varphi(q_p)}(\mathbb{Z}_p)
\]
is called the \emph{Teichm\"uller character}, which can be extended multiplicatively to a map
\[
\mathbb{Q}_p^\times \longrightarrow \mathbb{Q}_p^\times
\]
by setting
\[
\omega(x):=p^{v_p(x)}\omega\left(\frac{x}{p^{v_p(x)}}\right).
\]
We define 
\[ 
\langle x\rangle:=\frac{x}{\omega(x)} 
\]
for $x\in\mathbb{Q}_p^\times$. 
Both $\omega(x)$ and $\langle x \rangle$ are multiplicative on $\mathbb{Q}_p^\times$. 
Moreover, we have $\langle x \rangle \in 1+q_p\mathbb{Z}_p$ for any $x \in \mathbb{Q}_p^\times$.

For any $x \in \mathbb{Q}_p$ with $|x|_p \geqslant q_p$, and for any $s \in \mathbb{C}_p \setminus \{ 1 \}$ with $|s|_p < q_pp^{-1/(p-1)}$, we define
\begin{equation}\label{def_Hur}
	\zeta_p(s,x) := \frac{1}{s-1} \int_{\mathbb{Z}_p} \langle t+x \rangle^{1-s} \,\mathrm{d}t.
\end{equation} 
The function
\[ 
\zeta_p(\cdot,x) \colon \left\{ s \in \mathbb{C}_p \setminus \{ 1\} ~\mid~  |s|_p < q_pp^{-1/(p-1)}  \right\} \longrightarrow \mathbb{C}_p 
\]
is called the \emph{$p$-adic Hurwitz zeta function}. 
It is known that $\zeta_p(\cdot,x)$ is a $p$-adic meromorphic function on the region $|s|_p < q_pp^{-1/(p-1)}$, as stated in \cite[Proposition 11.2.8]{Coh2007}.

Let $\chi$ be a Dirichlet character of conductor $D$. 
Let $m$ be a common multiple of $D$ and $q_p$. 
The \emph{Kubota--Leopoldt $p$-adic $L$-function} is defined by
\begin{equation}\label{def_Kub-Leo}
	L_p(s,\chi) := \frac{\langle m \rangle^{1-s}}{m} \sum_{1 \leqslant \nu \leqslant m \atop p \nmid \nu} \chi(\nu) \zeta_p\left(s,\frac{\nu}{m} \right) 
\end{equation}
for $s \in \mathbb{C}_p \setminus \{ 1 \}$ with $|s|_p < q_pp^{-1/(p-1)}$. 
It is known that the right-hand side of \eqref{def_Kub-Leo} is independent of $m$ (see \cite[Proposition 11.3.8(1)]{Coh2007}).  

\begin{definition}\label{def_p-adic-zeta}
	For any integer $s \geqslant 2$, we define the \emph{$p$-adic zeta value} $\zeta_p(s)$ by 
	\[ 
	\zeta_p(s) := L_p(s,\omega^{1-s}). 
	\]
\end{definition}

We remark that
\[
\zeta_p(s) = \lim_{k \rightarrow s ~p\text{-adically}\atop k \in \mathbb{Z}_{<0},~~ k\equiv s \pmod{p-1}} \zeta(k) \quad \in \mathbb{Q}_p
\]
for any integer $s \geqslant 2$ (see \cite[Lemma 2.4]{Cal2005}). Moreover, $\zeta_p(s)$ vanishes when $s$ is a positive even integer.

For our purposes, we need the following lemmas.

\begin{lemma}\label{integral_of_shift}
	Let $i,l$ be integers such that $i \geqslant 2$ and $l \geqslant 2$. 
	Let $\nu$ be an integer such that $p \nmid \nu$. 
	Let $k$ be any nonnegative integer. 
	Then, we have
	\[  
	- \int_{\mathbb{Z}_p} \frac{1}{(1-i)(t+\nu/p^{l}+k)^{i-1}} \,\mathrm{d}t = \omega\left( \frac{\nu}{p^{l}} \right)^{1-i} \zeta_p\left(i, \frac{\nu}{p^l} \right)  - \sum_{\ell = 0}^{k-1} \frac{1}{(\ell+\nu/p^{l})^{i}}.  
	\]
	\textup{(}Note: When $k=0$, the empty sum $\sum_{\ell=0}^{k-1}$ is understood as $0$.\textup{)}
\end{lemma}

\begin{proof}
	By Lemma \ref{lem_Vol_trans} and Equation \eqref{def_Hur}, we have
	\begin{align*}
		&- \int_{\mathbb{Z}_p} \frac{1}{(1-i)(t+\nu/p^{l}+k)^{i-1}} \,\mathrm{d}t \\
		=~&- \int_{\mathbb{Z}_p} \frac{1}{(1-i)(t+\nu/p^{l})^{i-1}} \,\mathrm{d}t  - \sum_{\ell = 0}^{k-1} \frac{1}{(\ell+\nu/p^{l})^{i}} \\
		=~&\omega\left( \frac{\nu}{p^{l}} \right)^{1-i} \zeta_p\left(i, \frac{\nu}{p^l} \right)  - \sum_{\ell = 0}^{k-1} \frac{1}{(\ell+\nu/p^{l})^{i}}. 
	\end{align*}
\end{proof}

\begin{lemma}\label{lem_Hur_to_zeta}
	Let $i,l$ be integers such that $i \geqslant 2$ and $l \geqslant 2$. 
	Then, we have
	\[ 
	p^{li} \cdot \zeta_p(i) = \sum_{1 \leqslant \nu \leqslant p^{l} \atop p \nmid \nu} \omega\left( \frac{\nu}{p^l} \right)^{1-i}
	\zeta_p\left( i, \frac{\nu}{p^{l}} \right).  
	\]
\end{lemma}

\begin{proof}
	It follows from Definition \ref{def_p-adic-zeta} and Equation \eqref{def_Kub-Leo} (with $m=p^{l}$).
\end{proof}

\section{Setup for the proof for the $p$-adic case}\label{Sect_sett_p}

In this section, we will first set up the parameters. 
Then, we will construct a sequence of rational functions $V_n(t)$ and utilize their primitive functions $W_n(t)$ to produce linear forms in $1$ and $p$-adic odd zeta values. 

\subsection{Parameters}

Fix a positive integer $M$, and fix a finite collection of nonnegative integers $\{\delta_j\}_{j=1}^{J}$ such that
\begin{equation}\label{condition_on_deltas_p}
	0 \leqslant \delta_1 \leqslant \delta_2 \leqslant \cdots \leqslant \delta_J < \frac{M}{2},
\end{equation}
where $J \geqslant 1$ denotes the number of $\delta_j$'s. 

Let $s$ be a sufficiently large multiple of $2J$. 
Define the `power parameter' $l$ as
\begin{equation}\label{defi_l}
	l := \left\lfloor  \frac{\log s - 2\log\log s}{\log p} \right\rfloor.
\end{equation}
We assume that $s$ is sufficiently large so that $l \geqslant 2$. 
Note that $l$ is defined in such a way that $p^{l}$ is the largest power of $p$ not exceeding $s/\log^2 s$. 
The role of the parameter $l$ in the $p$-adic case is similar to that of the `numerator length parameter' $r$ in the classical case.

\subsection{Rational functions and their primitive functions}

For any $n \in \mathbb{Z}_{>0}$, we define
\begin{equation}\label{defi_L_n}
	L_n :=  \left\lfloor  \frac{\log(M\varphi(p^l)n)}{\log p}  \right\rfloor + 2.
\end{equation}  
Let $w_n$ be the unique integer satisfying 
\begin{align}
	w_n &\equiv s-1 + \left( -M\varphi(p^l)+ \frac{s}{J}\sum_{j=1}^{J}(M-2\delta_j) \right)n \pmod{(p-1)p^{L_n}}, \label{w_n_prop_cong} \\
	0 &\leqslant w_n < (p-1)p^{L_n}. \notag
\end{align}
Note that 
\begin{equation}\label{w_n_est}
	0 \leqslant w_n  < p^{l+3}Mn.
\end{equation}
The definitions of $L_n$ and $w_n$ are technical. 
They are used in Section \ref{Sect_p_adic_est} to prove a certain nonvanishing property (see Lemma \ref{lem_p_adic_order_of_sum_W}).

For any $n \in \mathbb{Z}_{>0}$, we define the following rational function $V_n(t)$:
\begin{align}
	V_n(t) := &~ p^{(l+1/(p-1))M\varphi(p^{l})n} \cdot \frac{\prod_{j=1}^{J} ((M-2\delta_j)n)!^{s/J}}{n!^{M\varphi(p^{l})}} \cdot \frac{1}{n!^{\lfloor w_n / n\rfloor}(w_n-\lfloor w_n / n\rfloor n)!} \notag \\
	& \times (t-w_n)_{w_n}  \cdot \frac{\prod_{1 \leqslant \nu \leqslant p^{l} \atop p \nmid \nu} (t+\nu/p^{l})_{Mn}}{\prod_{j=1}^{J} (t+\delta_j n)_{(M-2\delta_j)n+1}^{s/J}}.  \label{defi_V_n(t)}
\end{align}
Since $J \mid s$ and $(p-1) \mid \varphi(p^l)$, we have $V_n(t) \in \mathbb{Q}(t)$. 
Note that the degree of the rational function $V_n(t)$ is
\begin{align}
	\deg V_n &= w_n + M\varphi(p^{l})n -s - \frac{s}{J}\sum_{j=1}^{J} (M-2\delta_j)n \label{degVn}\\
	&< p^{l+3}Mn + M\varphi(p^{l})n -s - sn \quad(\text{by}\ \eqref{w_n_est}\ \text{and}\ \eqref{condition_on_deltas_p}) \notag\\
	&< \frac{p^3Msn}{\log^2 s} + \frac{Msn}{\log^2 s} - s - sn. \quad(\text{by}\ \eqref{defi_l})\notag
\end{align}
We assume that $s$ is sufficiently large so that
\begin{equation}\label{degVnleq-2}
	\deg V_n  \leqslant -2 \quad\text{for any~} n \in \mathbb{Z}_{>0}.
\end{equation}
Therefore, the rational function $V_n(t)$ has the following partial-fraction decomposition:
\begin{equation}\label{defi_bnik}
	V_n(t) =: \sum_{i=1}^{s}\sum_{k=\delta_1 n}^{(M-\delta_1)n} \frac{b_{n,i,k}}{(t+k)^i},
\end{equation} 
where the coefficients $b_{n,i,k} \in \mathbb{Q}$ ($1 \leqslant i \leqslant s \ \text{and}\ \delta_1n \leqslant k \leqslant (M-\delta_1)n$) are uniquely determined by $V_n(t)$. 

For any $n \in \mathbb{Z}_{>0}$, we define the function $W_n(t)$ by
\begin{equation}\label{defi_W_n(t)}
	W_n(t) := \sum_{k=\delta_1 n}^{(M-\delta_1)n} b_{n,1,k} \log_p\left\langle t+k \right\rangle  + \sum_{i=2}^{s} \sum_{k=\delta_1 n}^{(M-\delta_1)n} \frac{b_{n,i,k}}{(1-i)(t+k)^{i-1}}.
\end{equation}
The function $W_n(t)$ has has the following properties, which motivate its definition.

\begin{lemma}\label{prop_of_W_n(t)}
	For any $n \in \mathbb{Z}_{>0}$, the function $W_n(t)$ has the following properties.
	\begin{enumerate}
		\item[\textup{(1)}] $W_n(t)$ is a primitive function of $V_n(t)$ on $\mathbb{Q}_p \setminus \left( [-(M-\delta_1)n, -\delta_1n] \cap \mathbb{Z} \right)$\textup{;} that is,
		\[ 
		W_n^{\prime}(t) = V_n(t), \quad t \in \mathbb{Q}_p \setminus \left( [-(M-\delta_1)n, -\delta_1n] \cap \mathbb{Z} \right). 
		\]
		
		\item[\textup{(2)}] For any integer $\nu \in [1,p^{l}]$ with $p \nmid \nu$, the restriction of $W_n(t+\nu/p^{l})$ to $\mathbb{Z}_p$ is an overconvergent power series\textup{;} that is,
		\[ 
		W_n\left( t+\frac{\nu}{p^{l}} \right)  \in C^\dagger(\mathbb{Z}_p,\mathbb{Q}_p). 
		\]
		
		\item[\textup{(3)}] For $|t|_p \geqslant 1$, the function $W_n(t/p^{l})$ can be expressed as a $\mathbb{Q}_p$-coefficient convergent power series in $t^{-1}$ with vanishing constant term.
	\end{enumerate}
\end{lemma}

\begin{proof}
	(1). By \eqref{defi_W_n(t)} and \eqref{defi_bnik}, we have
	\begin{equation*}
		W_n^{\prime}(t) = \sum_{k=\delta_1 n}^{(M-\delta_1)n} \frac{b_{n,1,k}}{t+k}  + \sum_{i=2}^{s} \sum_{k=\delta_1 n}^{(M-\delta_1)n} \frac{b_{n,i,k}}{(t+k)^{i}} = V_n(t)
	\end{equation*}
	for $t \in \mathbb{Q}_p \setminus \left( [-(M-\delta_1)n, -\delta_1n] \cap \mathbb{Z} \right)$.
	
	(2). For $t \in \mathbb{Z}_p$, we have
	\begin{align*}
		\left\langle t+\frac{\nu}{p^{l}} + k  \right\rangle &= \left\langle \frac{\nu}{p^l} + k \right\rangle \langle 1+(\nu+p^{l}k)^{-1} p^{l}t \rangle \\
		&= \left\langle \frac{\nu}{p^l} + k \right\rangle ( 1+(\nu+p^{l}k)^{-1} p^{l}t ). 
	\end{align*}
	Thus, by \eqref{defi_W_n(t)}, we have
	\begin{align*}
		&~ W_n\left(t+\frac{\nu}{p^l}\right) \\
		=&~ \underbrace{\sum_{k=\delta_1 n}^{(M-\delta_1)n} b_{n,1,k}\log_p\left\langle \frac{\nu}{p^l} + k \right\rangle}_{\text{constant}}  + \sum_{k=\delta_1 n}^{(M-\delta_1)n} b_{n,1,k}\underbrace{\log_p( 1+(\nu+p^{l}k)^{-1} p^{l}t)}_{\in C^\dagger(\mathbb{Z}_p,\mathbb{Q}_p)}  \\
		&+ \sum_{i=2}^{s} \sum_{k=\delta_1 n}^{(M-\delta_1)n}  \frac{b_{n,i,k}}{(1-i)(\nu/p^l + k)^{i-1}} \underbrace{( 1+(\nu+p^{l}k)^{-1} p^{l}t)^{1-i}}_{\in C^\dagger(\mathbb{Z}_p,\mathbb{Q}_p)}.
	\end{align*}
	Every summand on the right-hand side is an overconvergent power series. 
	Therefore, $W_n(t+\nu/p^l) \in C^\dagger(\mathbb{Z}_p,\mathbb{Q}_p)$.
	
	(3). For $|t|_p \geqslant 1$, we have
	\[ 
	\left\langle \frac{t}{p^l} + k \right\rangle = \langle t + p^lk \rangle = \langle t \rangle (1+ p^{l}k t^{-1}). 
	\]
	Thus, by \eqref{defi_W_n(t)}, we have
	\begin{align*}
		W_n\left(\frac{t}{p^l}\right) =&~ \left(\sum_{k=\delta_1 n}^{(M-\delta_1)n} b_{n,1,k}\right) \log_p\langle t \rangle + \sum_{k=\delta_1 n}^{(M-\delta_1)n} b_{n,1,k} \log_p(1+ p^{l}k t^{-1}) \\
		&+ \sum_{i=2}^{s} \sum_{k=\delta_1 n}^{(M-\delta_1)n} \frac{p^{l(i-1)}b_{n,i,k}}{(1-i)} t^{1-i} (1+p^l k t^{-1})^{1-i}.
	\end{align*}
	By \eqref{defi_bnik} and \eqref{degVnleq-2}, we have
	\begin{equation}\label{sigma_1=0}
		\sum_{k=\delta_1 n}^{(M-\delta_1)n} b_{n,1,k} = \lim_{|t|_p \to +\infty} tV_n(t) = 0. 
	\end{equation}
	Therefore, for $|t|_p \geqslant 1$, 
	\begin{align*}
		W_n\left(\frac{t}{p^l}\right) =&~ \sum_{k=\delta_1 n}^{(M-\delta_1)n} b_{n,1,k} \underbrace{\log_p(1+ p^{l}k t^{-1})}_{\in \mathbb{Q}_p\llbracket p^{l}t^{-1}\rrbracket} \\
		&+ \sum_{i=2}^{s} \sum_{k=\delta_1 n}^{(M-\delta_1)n} \frac{p^{l(i-1)}b_{n,i,k}}{(1-i)} t^{1-i} \underbrace{(1+p^l k t^{-1})^{1-i}}_{\in \mathbb{Q}_p\llbracket p^{l}t^{-1}\rrbracket}
	\end{align*}
	is a $\mathbb{Q}_p$-coefficient convergent power series in $t^{-1}$ with vanishing constant term.
\end{proof}

\subsection{Linear forms}

By Lemma \ref{prop_of_W_n(t)}(2), for any integer $\nu \in [1,p^l]$ with $p \nmid \nu$, we have $W_n\left( t+\nu/p^{l} \right) \in C^\dagger(\mathbb{Z}_p,\mathbb{Q}_p)$. 
Therefore, $W_n\left( t+\nu/p^{l} \right)$ is Volkenborn integrable by Lemma \ref{lem_Vol_t^k}.

\begin{definition}\label{def_T_n}
	For any $n \in \mathbb{Z}_{>0}$, for any integer $\nu \in [1,p^l]$ with $p \nmid \nu$, we define
	\begin{equation}\label{def_Tntheta}
		T_{n,\nu/p^l} :=  -\int_{\mathbb{Z}_p} W_n\left(t+\frac{\nu}{p^l}\right)\mathrm{d}t. 
	\end{equation}
	Moreover, we define
	\begin{equation}\label{def_Tn}
		T_n:= \sum_{1 \leqslant \nu \leqslant p^l \atop p \nmid \nu} T_{n,\nu/p^l}.
	\end{equation}
\end{definition}

We define the following coefficients: 
\begin{align}
	\sigma_{n,i} &:= \sum_{k=\delta_1n}^{(M-\delta_1)n} b_{n,i,k}, \quad i=1,2,3,\ldots,s, \label{defi_sigma_i} \\
	\sigma_{n,0,\nu/p^l} &:= -\sum_{i=1}^{s}\sum_{k=\delta_1n}^{(M-\delta_1)n}\sum_{\ell=0}^{k-1} \frac{b_{n,i,k}}{(\ell+\nu/p^l)^i}, \quad 1 \leqslant \nu \leqslant p^l \text{~with~} p \nmid \nu. \label{defi_sigma_0_theta} \\
	\sigma_{n,0} &:= \sum_{1 \leqslant \nu \leqslant p^l \atop p \nmid \nu} \sigma_{n,0,\nu/p^l}. \label{defi_sigma_0}
\end{align}
(Note: When $k=0$, the empty sum $\sum_{\ell=0}^{k-1}$ in \eqref{defi_sigma_0_theta} is understood to be $0$.)

It turns out that $T_{n}$ is a linear form in $1$ and $p$-adic odd zeta values.

\begin{lemma}\label{lem_linear_form_p}
	For any $n \in \mathbb{Z}_{>0}$, $T_n$ is a linear form in $1$ and $p$-adic odd zeta values, given by
	\[ 
	T_n = \sigma_{n,0} + \sum_{3 \leqslant i \leqslant s-1 \atop i \text{~odd}} \sigma_{n,i} p^{li} \zeta_p(i), 
	\]
	where the coefficients $\sigma_{n,i}$ \textup{(}$i=3,5,\ldots,s-1$\textup{)} are defined by \eqref{defi_sigma_i} and $\sigma_{n,0}$ is defined by \eqref{defi_sigma_0}.
\end{lemma}

\begin{proof}
	For any integer $\nu \in [1,p^l]$ with $p \nmid \nu$, by \eqref{def_Tntheta} and \eqref{defi_W_n(t)}, we have
	\begin{align}
		T_{n,\nu/p^l} =  -& \sum_{k=\delta_1 n}^{(M-\delta_1)n} b_{n,1,k} \int_{\mathbb{Z}_p} \log_p\left\langle t+\frac{\nu}{p^l}+k \right\rangle \mathrm{d}t \notag\\
		-& \sum_{i=2}^{s} \sum_{k=\delta_1 n}^{(M-\delta_1)n} b_{n,i,k} \int_{\mathbb{Z}_p}\frac{1}{(1-i)(t+\nu/p^l+k)^{i-1}} \,\mathrm{d}t. \label{lin_T_n_theta_equation_1}
	\end{align}
	By Lemma \ref{integral_of_shift}, we have
	\begin{equation}\label{lin_T_n_theta_equation_2}
		- \int_{\mathbb{Z}_p} \frac{1}{(1-i)(t+\nu/p^{l}+k)^{i-1}} \,\mathrm{d}t = \omega\left( \frac{\nu}{p^{l}} \right)^{1-i} \zeta_p\left(i, \frac{\nu}{p^l} \right)  - \sum_{\ell = 0}^{k-1} \frac{1}{(\ell+\nu/p^{l})^{i}}.
	\end{equation}
	By Lemma \ref{lem_Vol_trans}, we have
	\begin{equation}\label{lin_T_n_theta_equation_3}
		-\int_{\mathbb{Z}_p} \log_p\left\langle t+\frac{\nu}{p^l}+k \right\rangle \mathrm{d}t = -\int_{\mathbb{Z}_p} \log_p\left\langle t+\frac{\nu}{p^l} \right\rangle \mathrm{d}t - \sum_{\ell=0}^{k-1} \frac{1}{\ell+\nu/p^l}.
	\end{equation}
	Substituting \eqref{lin_T_n_theta_equation_2} and \eqref{lin_T_n_theta_equation_3} into \eqref{lin_T_n_theta_equation_1}, we obtain
	\begin{align*}
		T_{n,\nu/p^l} = &-\sigma_{n,1} \int_{\mathbb{Z}_p} \log_p\left\langle t+\frac{\nu}{p^l} \right\rangle \mathrm{d}t \\
		&+\sigma_{n,0,\nu/p^l} + \sum_{i=2}^{s} \sigma_{n,i} \omega\left( \frac{\nu}{p^{l}} \right)^{1-i} \zeta_p\left(i, \frac{\nu}{p^l} \right).
	\end{align*}
	Since $\sigma_{n,1} = 0$ by \eqref{sigma_1=0}, we have
	\begin{equation}\label{lin_T_n_theta_equation_4}
		T_{n,\nu/p^l} = \sigma_{n,0,\nu/p^l} + \sum_{i=2}^{s} \sigma_{n,i} \omega\left( \frac{\nu}{p^{l}} \right)^{1-i} \zeta_p\left(i, \frac{\nu}{p^l} \right).
	\end{equation}
	Finally, substituting \eqref{lin_T_n_theta_equation_4} into \eqref{def_Tn} and using Lemma \ref{lem_Hur_to_zeta}, we obtain
	\begin{equation*}
		T_n =  \sigma_{n,0} + \sum_{i=2}^{s} \sigma_{n,i} p^{li} \zeta_p(i).
	\end{equation*}
	Since $\zeta_p(i) =0 $ for any positive even integer $i$, this completes the proof of Lemma \ref{lem_linear_form_p}.
\end{proof}

\section{Arithmetic properties for the $p$-adic case}\label{Sect_arit_p}

Recall that $D_{m} = \operatorname{lcm}[1,2,3,\ldots,m]$. 
The following lemma is analogous to Lemma \ref{lem_arith_anik}:

\begin{lemma}\label{lem_arith_bnik}
	For any integer $n > s^2$, we have
	\[ 
	\Phi_n^{-s/J} D_{(M-2\delta_1)n}^{s-i} \cdot  b_{n,i,k} \in \mathbb{Z} \quad (1 \leqslant i \leqslant s\ \text{and}\ \delta_1n \leqslant k \leqslant (M-\delta_1)n),   
	\]
	where the factor $\Phi_n$ is given by \eqref{defi_Phi} and \eqref{defi_phi}.
\end{lemma}

\begin{proof}
	The proof is similar to that of Lemma \ref{lem_arith_anik}. 
	We indicate the modifications.
	
	Fix any $i \in \{1,2,\ldots,s\}$ and any $k \in [\delta_1n,~(M-\delta_1)n] \cap \mathbb{Z}$.
	
	By \eqref{defi_bnik}, we have
	\begin{equation*}
		b_{n,i,k} = \frac{1}{(s-i)!} \left( V_n(t)(t+k)^{s} \right)^{(s-i)} \Big|_{t=-k}.
	\end{equation*}
	
	Define the elementary bricks:
	\begin{align*}
		F_{1,\tau}(t) &= \frac{(t-w_n +\tau n)_n}{n!}, \quad \tau = 0,1,2,\ldots,\lfloor w_n/n\rfloor - 1, \\
		F_{1,\lfloor w_n/n\rfloor}(t) &=  \frac{(t-w_n+\lfloor w_n/n\rfloor n)_{w_n - \lfloor w_n/n\rfloor n}}{(w_n - \lfloor w_n/n\rfloor n)!}, \\
		F_{\nu/p^l,\tau}(t) &= p^{ln} \cdot p^{\lfloor n/(p-1)\rfloor} \cdot \frac{(t+\nu/p^l+\tau n)_n}{n!}, \\
		&\qquad  1 \leqslant \nu \leqslant p^l \text{~with~} p \nmid \nu, \quad \tau=0,1,2,\ldots,M-1, \\
		G_j(t) &=  \frac{((M-2\delta_j)n)!}{(t+\delta_jn)_{(M-2\delta_j)n+1}}, \quad j=1,2,\ldots,J.
	\end{align*}
	Then, by \eqref{defi_V_n(t)}, we have
	\[ 
	V_n(t)(t+k)^s = A \prod_{\tau=0}^{\lfloor w_n/n\rfloor} F_{1,\tau}(t) \cdot \prod_{1 \leqslant \nu \leqslant p^l \atop p \nmid \nu }\prod_{\tau=0}^{M-1} F_{\nu/p^l,\tau}(t) \cdot \prod_{j=1}^{J} \left(G_j(t)(t+k)\right)^{s/J}, 
	\]
	where $A = p^{M\varphi(p^l)\left\{ n/(p-1) \right\}}$ is a positive integer. 
	By Lemma \ref{lem_F}, for any polynomial $F(t)$ of the form $F_{1,\tau}(t)$, $F_{\nu/p^l,\tau}(t)$, we have
	\begin{equation*}
		D_{n}^{\lambda} \cdot \frac{1}{\lambda!} F^{(\lambda)}(t) \big|_{t=-k} \in \mathbb{Z} 
	\end{equation*} 
	for any nonnegative integer $\lambda$.
	By \eqref{lem_G_1} of Lemma \ref{lem_G} (with $a_0 = \delta_1n$ and $b_0 = (M-\delta_1)n$), we have
	\begin{equation*}
		D_{(M-2\delta_1)n}^{\lambda} \cdot \frac{1}{\lambda!} \left( G_j(t)(t+k) \right)^{(\lambda)} \big|_{t=-k} \in \mathbb{Z} 
	\end{equation*}
	for any $j=1,2,\ldots,J$ and any nonnegative integer $\lambda$.
	
	The rest of the proof of Lemma \ref{lem_arith_bnik} is literally the same as the rest of the proof of Lemma \ref{lem_arith_anik}, except that we need to replace $a_{n,i,k}$ by $b_{n,i,k}$.
\end{proof}

The arithmetic property of $b_{n,i,k}$ propagates to the arithmetic property of $\sigma_{n,i}$ (defined by \eqref{defi_sigma_i} and \eqref{defi_sigma_0}), as the following lemma shows.

\begin{lemma}\label{lem_arith_sigmai}
	For any integer $n > s^2$, we have
	\[ 
	\Phi_n^{-s/J} D_{(M-2\delta_1)n}^{s-i} \cdot \sigma_{n,i} \in \mathbb{Z} \quad (i=1,2,\ldots,s), 
	\]
	and
	\[ 
	\Phi_n^{-s/J} D_{(M-2\delta_1)n}^{s} \cdot \sigma_{n,0} \in \mathbb{Z}. 
	\]
\end{lemma}

\begin{proof}
	The first assertion follows directly from \eqref{defi_sigma_i} and Lemma \ref{lem_arith_bnik}.
	
	To prove the second assertion, by \eqref{defi_sigma_0}, it suffices to prove that for any integer $\nu \in [1,p^l]$ with $p \nmid \nu$ we have
	\begin{equation}\label{arith_sigma0theta}
		\Phi_n^{-s/J} D_{(M-2\delta_1)n}^{s} \cdot \sigma_{n,0,\nu/p^l} \in \mathbb{Z}.
	\end{equation} 
	We prove \eqref{arith_sigma0theta} by contradiction. 
	Suppose that 
	\begin{equation}\label{not_true}
		\Phi_n^{-s/J} D_{(M-2\delta_1)n}^{s} \cdot \sigma_{n,0,\nu/p^l} \notin \mathbb{Z}.
	\end{equation}
	By substituting \eqref{defi_sigma_0_theta} into \eqref{not_true}, we see that there exist integers $k_0$, $\ell_0$ such that $\delta_1 n \leqslant k_0 \leqslant (M-\delta_1)n$, $0 \leqslant \ell_0 < k_0$, and
	\begin{equation}\label{not_true_2}
		\Phi_n^{-s/J} D_{(M-2\delta_1)n}^{s} \sum_{i=1}^{s} \frac{b_{n,i,k_0}}{(\ell_0+\nu/p^l)^{i}} \notin \mathbb{Z}.
	\end{equation}
	Write $\nu^{\prime} = p^l - \nu$. 
	Note that $t+k_0-\ell_0-\nu/p^l$ is a factor of $(t+\nu^{\prime}/p^l)_{Mn}$. 
	Thus, $t+k_0-\ell_0-\nu/p^l$ is a factor of the numerator of $V_n(t)$ (see \eqref{defi_V_n(t)}). We have
	\begin{equation}\label{V_n(t)_has_root}
		V_n\left(-k_0+\ell_0+\frac{\nu}{p^l} \right) = 0.
	\end{equation}
	Substituting \eqref{defi_bnik} into \eqref{V_n(t)_has_root}, and using \eqref{not_true_2}, we obtain
	\begin{align*}
		&~ \Phi_n^{-s/J} D_{(M-2\delta_1)n}^{s} \sum_{i=1}^{s} \frac{b_{n,i,k_0}}{(\ell_0+\nu/p^l)^{i}} \\
		=&~ -\Phi_n^{-s/J} D_{(M-2\delta_1)n}^{s} \sum_{i=1}^{s}\sum_{k=\delta_1n \atop k \neq k_0}^{(M-\delta_1)n} \frac{b_{n,i,k}}{(k-k_0+\ell_0+\nu/p^l)^i} \not\in\mathbb{Z}. 
	\end{align*}
	Therefore, there exist a prime $q$, two integers $i_0,i_1 \in \{1,2,\ldots,s\}$, and an integer $k_1 \in \{\delta_1 n,\ldots, (M-\delta_1)n\}$ with $k_1 \neq k_0$ such that
	\begin{align*}
		v_q\left( \Phi_n^{-s/J} D_{(M-2\delta_1)n}^{s} \cdot \frac{b_{n,i_0,k_0}}{(\ell_0+\nu/p^l)^{i_0}}  \right)  &< 0 \quad\text{and} \\ 
		v_q\left( \Phi_n^{-s/J} D_{(M-2\delta_1)n}^{s} \cdot \frac{b_{n,i_1,k_1}}{(k_1-k_0+\ell_0+\nu/p^l)^{i_1}}  \right) &< 0. 
	\end{align*}
	On the other hand, by Lemma \ref{lem_arith_bnik}, we have
	\[ 
	v_q\left( \Phi_n^{-s/J} D_{(M-2\delta_1)n}^{s-i_0} \cdot b_{n,i_0,k_0} \right) \geqslant 0 \quad\text{and}\quad v_q\left( \Phi_n^{-s/J} D_{(M-2\delta_1)n}^{s-i_1} \cdot b_{n,i_1,k_1}  \right) \geqslant 0. 
	\]
	Therefore, 
	\[ 
	v_q\left( \ell_0 + \frac{\nu}{p^l} \right) > v_q(D_{(M-2\delta_1)n}) \quad\text{and}\quad v_q\left( k_1-k_0+\ell_0+\frac{\nu}{p^l} \right) > v_q(D_{(M-2\delta_1)n}). 
	\]
	It follows that $v_q(k_1-k_0) > v_q(D_{(M-2\delta_1)n})$, which is a contradiction because $0<|k_1-k_0| \leqslant (M-2\delta_1)n$. 
	The proof of Lemma \ref{lem_arith_sigmai} is complete.
\end{proof}

\section{Archimedean estimates}\label{Sect_arch_est}

In this section we will estimate the Archimedean norm of the coefficients $\sigma_{n,i}$ ($i=0,1,2,\ldots,s$) defined by \eqref{defi_sigma_0} and \eqref{defi_sigma_i}. 

\begin{lemma}\label{lem_ana_sigma}
	We have 
	\begin{equation}\label{sigma_estimate}
		\max_{0 \leqslant i \leqslant s} |\sigma_{n,i}| \leqslant \exp\left( \beta_p(s)n + o(n) \right) \quad \text{as~} n \to \infty, 
	\end{equation}
	where $\beta_p(s)$ is a constant independent of $n$. Moreover, we have
	\begin{equation}\label{beta_p}
		\beta_p(s) \sim \log2 \cdot \frac{s}{J} \sum_{j=1}^{J} (M - 2\delta_j) \quad\text{as~} s \to \infty. 
	\end{equation}
\end{lemma}

\begin{proof}
	By \eqref{defi_sigma_i}--\eqref{defi_sigma_0}, we have
	\begin{equation}\label{sigma<b}
		\max_{0 \leqslant i \leqslant s} |\sigma_{n,i}| \leqslant (Mn+1)^{2}sp^{l(s+1)} \cdot \max_{i,k} |b_{n,i,k}|, 
	\end{equation}
	where $\max_{i,k}$ is taken over all $i \in \{ 1,2,\ldots,s \}$ and $k \in [\delta_1n, (M-\delta_1)n] \cap \mathbb{Z}$. 
	
	Now fix any pair $(i,k)$ in the range above. 
	By \eqref{defi_bnik} and Cauchy's integral formula, we have
	\[ 
	b_{n,i,k} = \frac{1}{2\pi \sqrt{-1}} \int_{|z+k|=1/p^{l+1}} (z+k)^{i-1}V_n(z)\,\mathrm{d}z. 
	\]
	Therefore, we have (recall \eqref{defi_V_n(t)})
	\begin{align}
		|b_{n,i,k}| &\leqslant \max_{|z+k| = 1/p^{l+1}} |V_n(z)| \notag\\
		=&~ p^{(l+1/(p-1))M\varphi(p^{l})n} \cdot \frac{\prod_{j=1}^{J} ((M-2\delta_j)n)!^{s/J}}{n!^{M\varphi(p^{l})}} \cdot \frac{1}{n!^{\lfloor w_n / n\rfloor}(w_n-\lfloor w_n / n\rfloor n)!} \notag \\
		& \times \max_{|z+k|=1/p^{l+1}} |(z-w_n)_{w_n}| \frac{\prod_{1 \leqslant \nu \leqslant p^{l} \atop p \nmid \nu} |(z+\nu/p^{l})_{Mn}|}{\prod_{j=1}^{J} |(z+\delta_j n)_{(M-2\delta_j)n+1}|^{s/J}}. \label{|b|<maxV}
	\end{align}
	In the following, the complex number $z$ is always assumed to be on the circle $|z+k|=1/p^{l+1}$. 
	We will use the triangle inequality to bound \eqref{|b|<maxV} from above. 
	
	We have
	\begin{align}
		&~ \max_{|z+k|=1/p^{l+1}} \frac{|(z-w_n)_{w_n}|}{n!^{\lfloor w_n / n\rfloor}(w_n-\lfloor w_n / n\rfloor n)!} \notag\\
		\leqslant&~  \frac{(Mn+2)(Mn+3)\cdots(Mn+w_n+1)}{n!^{\lfloor w_n / n\rfloor}(w_n-\lfloor w_n / n\rfloor n)!} \notag\\
		=&~ \frac{(Mn+w_n+1)!}{(Mn+1)!n!^{\lfloor w_n / n\rfloor}(w_n-\lfloor w_n / n\rfloor n)!} \notag\\
		\leqslant&~ \left( \underbrace{1+1+\cdots+1}_{\lfloor w_n/n \rfloor + 2} \right)^{Mn+w_n+1} \notag\\
		\leqslant&~    (p^{l+3}M+2)^{Mn+p^{l+3}Mn + 1}. \quad(\text{by}\ \eqref{w_n_est}) \label{p_est1}
	\end{align}
	For any integer $\nu \in [1,p^l]$ with $p \nmid \nu$, we have
	\begin{align}
		\max_{|z+k|=1/p^{l+1}} \frac{|(z+\nu/p^{l})_{Mn}|}{n!^M} &\leqslant \frac{(k+1)!(Mn+1-k)!}{n!^M} \notag\\
		&\leqslant \frac{(Mn+2)!}{n!^M} = (Mn+2)(Mn+1) \cdot \frac{(Mn)!}{n!^M} \notag\\
		& \leqslant 10M^2n^2 \cdot M^{Mn}. \label{p_est2}
	\end{align}
	Arguing as in the proof of \eqref{est3}, we obtain
	\begin{equation}\label{p_est3}
		\max_{|z+k| = 1/p^{l+1}} \frac{((M-2\delta_j)n)!}{|(z+\delta_jn)_{(M-2\delta_j)n+1}|} \leqslant p^{3(l+1)}M^2n^2 \cdot 2^{(M-2\delta_j)n}.
	\end{equation}
	Combining inequalities \eqref{sigma<b}--\eqref{p_est3}, we obtain
	\[ 
	\max_{0 \leqslant i \leqslant s} |\sigma_{n,i}| \leqslant \exp\left( \beta_p(s)n + o(n) \right) \quad \text{as~} n \to \infty,  
	\]
	where
	\begin{align*}
		\beta_p(s) =&~ \left( l+\frac{1}{p-1} \right)M\varphi(p^l)\log p + M\varphi(p^l)\log M \\
		&+ (p^{l+3}+1)M\log(p^{l+3}M+2) + \log 2 \cdot \frac{s}{J} \sum_{j=1}^{J} (M - 2\delta_j). 
	\end{align*}
	
	By \eqref{defi_l}, we have
	\[ 
	\varphi(p^l)<p^l \leqslant \frac{s}{\log^2 s} \quad\text{and}\quad l\log p < \log s. 
	\]
	Therefore, 
	\[ \beta_p(s) \sim \log2 \cdot \frac{s}{J} \sum_{j=1}^{J} (M - 2\delta_j) \quad\text{as~} s \to \infty.  \qedhere
	\]
\end{proof}

\section{$p$-Adic estimates}\label{Sect_p_adic_est}
In this section, we will determine the $p$-adic norm of $T_n$ (defined by \eqref{def_Tn}). 

By Equation \eqref{def_Tntheta}, Lemma \ref{lem_L_1}, and Lemma \ref{prop_of_W_n(t)}(1)(2), we have
\begin{equation}\label{T_ntheta_L1}
	T_{n,\nu/p^l} = - \mathcal{L}_1\left( V_n\left(t+\frac{\nu}{p^l}\right) \right) - W_n\left( \frac{\nu}{p^l} \right)
\end{equation}
for any integer $\nu \in [1,p^l]$ with $p \nmid \nu$. 
Substituting \eqref{T_ntheta_L1} into \eqref{def_Tn}, we have
\begin{equation}\label{T_n_L1}
	T_n = -\sum_{1 \leqslant \nu \leqslant p^l \atop p \nmid \nu}  \mathcal{L}_1\left( V_n\left(t+\frac{\nu}{p^l}\right) \right) - \sum_{1 \leqslant \nu \leqslant p^l \atop p \nmid \nu} W_n\left( \frac{\nu}{p^l} \right).
\end{equation}

In the rest of this section, we will show that
\begin{itemize}
	\item For any sufficiently large $n \in \mathbb{Z}_{>0}$, we can explicitly determine the $p$-adic order of $\sum_{\nu} W_n(\nu/p^l)$.
	\item For any sufficiently large $n \in \mathbb{Z}_{>0}$,  the $p$-adic order of $\mathcal{L}_1\left( V_n\left(t+\nu/p^l\right) \right)$ is strictly larger than the $p$-adic order of $\sum_{\nu} W_n(\nu/p^l)$.
\end{itemize}
In this way, we can explicitly determine $|T_n|_p$ for any sufficiently large $n \in \mathbb{Z}_{>0}$. 

We need an elementary lemma: 
\begin{lemma}\label{elementary_lemma}
	For any $m \in \mathbb{Z}$, we have
	\[ 
	v_p\left( \sum_{1 \leqslant \nu \leqslant p^l \atop p \nmid \nu} \nu^{m} \right) \geqslant l-1. 
	\]
\end{lemma}

\begin{proof}
	For any positive integer $N$, for any integer $\nu \in[1,p^l]$ with $p \nmid \nu$, we have
	\[ 
	\nu^{N\varphi(p^l)+m} \equiv \nu^{m} \pmod{p^{l}\mathbb{Z}_p}. 
	\]
	Replacing $m$ by $N\varphi(p^l)+m$ for some sufficiently large positive integer $N$, we may assume without loss of generality that $m > 0$. 
	Then, the conclusion follows by an easy induction on $l$. 
\end{proof}

Now, we explicitly determine the $p$-adic order of $\sum_{\nu} W_n(\nu/p^l)$.

\begin{lemma}\label{lem_p_adic_order_of_sum_W}
	For any sufficiently large $n \in \mathbb{Z}_{>0}$, we have
	\begin{align}
		&~v_p\left( \sum_{1 \leqslant \nu \leqslant p^l \atop p \nmid \nu} W_n\left( \frac{\nu}{p^l} \right) \right) \notag\\ 
		=&~ \left( l + \frac{1}{p-1} \right)M\varphi(p^l)n + \frac{s}{J} \sum_{j=1}^{J} v_p(((M-2\delta_j)n)!) \notag\\ 
		&- M\varphi(p^l)v_p\left( n! \right) - \left\lfloor \frac{w_n}{n} \right\rfloor v_p(n!) - v_p\left( \left( w_n -\left\lfloor \frac{w_n}{n} \right\rfloor n \right)! \right) \notag\\
		&+ l|\deg V_n|  - v_p(|\deg V_n| - 1) - 1. \label{v_p(sumWn)}
	\end{align}
	\textup{(}The main term on the right-hand side is $l|\deg V_n|$\textup{)}. 
	In particular, we have
	\begin{equation}\label{p_adic_norm_of_sumWn}
		\exp\left( -\alpha_{p,1}(s)n + o(n) \right) \leqslant \left| \sum_{1 \leqslant \nu \leqslant p^l \atop p \nmid \nu} W_n\left( \frac{\nu}{p^l} \right) \right|_p \leqslant \exp\left( -\alpha_{p,2}(s)n + o(n) \right)
	\end{equation} 
	as $n \to \infty$, where
	\begin{align}
		\alpha_{p,1}(s) &= \frac{sl\log p}{J} \sum_{j=1}^{J} (M-2\delta_j) + \frac{s\log p}{(p-1)J}\sum_{j=1}^{J} (M-2\delta_j),  \label{defi_alpha_1}\\
		\alpha_{p,2}(s) &= \alpha_{p,1}(s) - \left( l+\frac{1}{p-1}\right)p^{l+3}M\log p. \label{defi_alpha_2}
	\end{align}
\end{lemma}

\begin{proof}
	By \eqref{defi_V_n(t)}, we have
	\begin{align}
		&~ V_n\left( \frac{t}{p^l} \right) \notag\\
		=&~ p^{(l+1/(p-1))M\varphi(p^{l})n} \cdot \frac{\prod_{j=1}^{J} ((M-2\delta_j)n)!^{s/J}}{n!^{M\varphi(p^{l})}} \cdot \frac{1}{n!^{\lfloor w_n / n\rfloor}(w_n-\lfloor w_n / n\rfloor n)!} \notag \\
		&\times p^{l|\deg V_n|} \cdot t^{-|\deg V_n|} \notag\\
		&\times \prod_{1 \leqslant \nu \leqslant p^l \atop p \nmid \nu}\prod_{\tau=0}^{Mn-1} \left( 1+(\nu+p^l\tau)t^{-1} \right) \label{PPP}\\
		&\times \frac{\prod_{\tau=1}^{w_n}\left( 1 -\tau p^{l} t^{-1} \right)}{\prod_{j=1}^{J}\prod_{\tau=0}^{(M-2\delta_j)n}\left( 1+(\delta_j n + \tau)p^{l} t^{-1} \right)^{s/J}}. \label{QQQ}
	\end{align}
	Note that the product on line \eqref{PPP} is a $\mathbb{Z}_p$-coefficient polynomial in $t^{-1}$ of degree $M\varphi(p^l)n$. 
	The product on line \eqref{QQQ} belongs to $\mathbb{Z}_p\llbracket p^{l}t^{-1} \rrbracket$. 
	Therefore, we can write $V_n(t/p^{l})$ as 
	\begin{align}
		&~ V_n\left( \frac{t}{p^l} \right) \notag\\
		=&~ p^{(l+1/(p-1))M\varphi(p^{l})n} \cdot \frac{\prod_{j=1}^{J} ((M-2\delta_j)n)!^{s/J}}{n!^{M\varphi(p^{l})}} \cdot \frac{1}{n!^{\lfloor w_n / n\rfloor}(w_n-\lfloor w_n / n\rfloor n)!} \notag \\
		&\times p^{l|\deg V_n|} \cdot t^{-|\deg V_n|} \sum_{k=0}^{\infty} h_k t^{-k}, \label{V_n(t/p^l)_in_t^-1}
	\end{align}
	where the coefficients $h_k \in \mathbb{Z}_p$ ($k=0,1,2,\ldots$) satisfy
	\begin{equation}\label{v_p(h_k)}
		v_p(h_k) \geqslant l \max\left\{ 0, k - M\varphi(p^l)n \right\}. 
	\end{equation} 
	
	By Lemma \ref{prop_of_W_n(t)}(3), the function $W_n(t/p^l)$ can be expressed as a $\mathbb{Q}_p$-coefficient convergent power series in $t^{-1}$ for $|t|_p \geqslant 1$ with vanishing constant term. 
	Moreover, by Lemma \ref{prop_of_W_n(t)}(1), we have
	\begin{equation}\label{Wnprime=Vn}
		\frac{\mathrm{d}}{\mathrm{d}t} \left( W_n\left( \frac{t}{p^l} \right) \right) = \frac{1}{p^l} V_n\left( \frac{t}{p^l} \right), \quad |t|_p \geqslant 1.
	\end{equation}
	Comparing \eqref{Wnprime=Vn} to \eqref{V_n(t/p^l)_in_t^-1}, we obtain
	\begin{align}
		&~\sum_{1 \leqslant \nu \leqslant p^l \atop p \nmid \nu} W_n\left( \frac{\nu}{p^l} \right) \notag\\
		=&~ p^{(l+1/(p-1))M\varphi(p^{l})n} \cdot \frac{\prod_{j=1}^{J} ((M-2\delta_j)n)!^{s/J}}{n!^{M\varphi(p^{l})}} \cdot \frac{1}{n!^{\lfloor w_n / n\rfloor}(w_n-\lfloor w_n / n\rfloor n)!} \notag\\
		&\times p^{l|\deg V_n|} \cdot \frac{1}{p^l} \notag \\
		&\times \sum_{k=0}^{\infty} \frac{h_k}{-|\deg V_n|-k+1} \sum_{1 \leqslant \nu \leqslant p^l \atop p \nmid \nu} \nu^{-|\deg V_n|-k+1}.  \label{expression_sumWn}
	\end{align}
	In the following, we will prove that for any sufficiently large $n \in \mathbb{Z}_{>0}$
	\begin{align}\label{claim}
		&~ v_p\left(  \frac{h_k}{-|\deg V_n|-k+1} \sum_{1 \leqslant \nu \leqslant p^l \atop p \nmid \nu} \nu^{-|\deg V_n|-k+1}  \right) \notag\\ 
		>&~ v_p\left(  \frac{h_0}{-|\deg V_n|+1} \sum_{1 \leqslant \nu \leqslant p^l \atop p \nmid \nu} \nu^{-|\deg V_n|+1} \right) \quad \text{for any~} k>0.
	\end{align}
	
	Note that $h_0 = 1$. 
	By \eqref{degVn} and \eqref{w_n_prop_cong}, we have
	\begin{equation}\label{1-|degVn|}
		-|\deg V_n| + 1 \equiv 0 \pmod{(p-1)p^{L_n}}. 
	\end{equation}
	Since $L_n > l$ by \eqref{defi_L_n}, we have $\varphi(p^{l}) \mid (-|\deg V_n|+1)$. Thus,
	\[ 
	\sum_{1 \leqslant \nu \leqslant p^l \atop p \nmid \nu} \nu^{-|\deg V_n|+1}  \equiv  \sum_{1 \leqslant \nu \leqslant p^l \atop p \nmid \nu} 1 \equiv p^{l-1}(p-1) \pmod{p^l\mathbb{Z}_p}. 
	\]
	Therefore,
	\begin{equation}\label{v_p(0th)}
		v_p\left( \frac{h_0}{-|\deg V_n|+1} \sum_{1 \leqslant \nu \leqslant p^l \atop p \nmid \nu} \nu^{-|\deg V_n|+1} \right) = l-1 -v_p\left( |\deg V_n| - 1 \right).
	\end{equation}
	
	For $k>0$, we consider two cases. 
	
	Case 1: $v_p(k) < v_p(|\deg V_n| - 1)$. 
	In this case, we have $v_p(-|\deg V_n|-k+1) = v_p(k) < v_p(|\deg V_n| - 1)$. 
	By \eqref{v_p(h_k)} and Lemma \ref{elementary_lemma}, we have $v_p(h_k) \geqslant 0$ and 
	\[ 
	v_p\left( \sum_{1 \leqslant \nu \leqslant p^l \atop p \nmid \nu} \nu^{-|\deg V_n|-k+1} \right) \geqslant l-1. 
	\]
	In conclusion, \eqref{claim} holds in Case 1.
	
	Case 2: $v_p(k) \geqslant v_p(|\deg V_n| - 1)$. 
	In this case, by \eqref{1-|degVn|} we have $v_p(k) \geqslant L_n$. 
	By $k>0$ and \eqref{defi_L_n}, we have
	\begin{equation}\label{k_is_large}
		k \geqslant p^{L_n} > p \cdot M\varphi(p^l)n \geqslant 2 \cdot M\varphi(p^l)n.  
	\end{equation}
	By \eqref{v_p(h_k)}, we have
	\begin{equation}\label{case2_eq1}
		v_p(h_k) \geqslant l \cdot (k - M\varphi(p^l)n) > 2 \cdot \frac{k}{2} = k.  
	\end{equation}
	Combining \eqref{case2_eq1} and Lemma \ref{elementary_lemma}, we have
	\begin{align*}
		&~ v_p\left(  \frac{h_k}{-|\deg V_n|-k+1} \sum_{1 \leqslant \nu \leqslant p^l \atop p \nmid \nu} \nu^{-|\deg V_n|-k+1}  \right) \\
		>&~ k - v_p(|\deg V_n|+k-1) + l - 1. 
	\end{align*}
	In view of \eqref{v_p(0th)}, to prove \eqref{claim}, it suffices to prove that
	\[ 
	k + v_p\left( |\deg V_n| - 1 \right) > v_p(|\deg V_n|+k-1). 
	\]
	Since $v_p\left( |\deg V_n| - 1 \right) \geqslant L_n >0$, it suffices to prove that
	\[ 
	k > \frac{\log(|\deg V_n|+k-1)}{\log p}, 
	\]
	or
	\begin{equation*}
		p^k > |\deg V_n|+k-1. 
	\end{equation*}
	By \eqref{degVn}, we have $|\deg V_n|+k-1 < Msn + k$ for sufficiently large $n$. 
	Note that $k > n$ by \eqref{k_is_large}. 
	If $n$ is large enough, we have
	\[ 
	2^k > Msn + k \quad \text{for any~} k > n. 
	\]
	Therefore, the claimed inequality \eqref{claim} holds for any sufficiently large $n$. 
	
	By \eqref{expression_sumWn}, \eqref{claim}, and \eqref{v_p(0th)}, we obtain an explicit formula for $v_p(\sum_{\nu} W_n(\nu/p^l))$ for any sufficiently large $n \in \mathbb{Z}_{>0}$, as stated in \eqref{v_p(sumWn)}. 
	Since $v_p(n!) = n/(p-1) + O(\log n)$ and $0 \leqslant w_n < p^{l+3}Mn$ (see \eqref{w_n_est}), Equation \eqref{v_p(sumWn)} implies Inequality \eqref{p_adic_norm_of_sumWn}. 
	The proof of Lemma \ref{lem_p_adic_order_of_sum_W} is complete.
\end{proof}

Next, we bound from below the $p$-adic order of $\mathcal{L}_1(V_n(t+\nu/p^l))$. 

\begin{lemma}\label{lem_v_p(L_1(V_n(shift)))}
	For any $n\in \mathbb{Z}_{>0}$, for any integer $\nu \in [1,p^l]$ with $p \nmid \nu$, we have
	\begin{align}
		&~ v_p\left( \mathcal{L}_1\left( V_n\left( t+\frac{\nu}{p^l} \right) \right) \right) \notag\\
		\geqslant& ~\left( l + \frac{1}{p-1} \right)M\varphi(p^l)n + \frac{s}{J} \sum_{j=1}^{J} v_p(((M-2\delta_j)n)!) \notag\\ 
		&- M\varphi(p^l)v_p\left( n! \right) - \left\lfloor \frac{w_n}{n} \right\rfloor v_p(n!) - v_p\left( \left( w_n -\left\lfloor \frac{w_n}{n} \right\rfloor n \right)! \right) \notag\\
		&+ l|\deg V_n|  +lMn -1-\frac{\log(Mn+1)}{\log p}. \label{v_p(L_1(V_n(shift)))}
	\end{align}
\end{lemma}

\begin{proof}
	By \eqref{defi_V_n(t)}, we have
	\begin{align}
		&~ V_n\left(t+\frac{\nu}{p^l}\right) \notag\\
		=&~ p^{(l+1/(p-1))M\varphi(p^{l})n} \cdot \frac{\prod_{j=1}^{J} ((M-2\delta_j)n)!^{s/J}}{n!^{M\varphi(p^{l})}} \cdot \frac{1}{n!^{\lfloor w_n / n\rfloor}(w_n-\lfloor w_n / n\rfloor n)!} \notag \\
		&\times p^{l|\deg V_n|} \cdot p^{lMn} \notag\\
		&\times \prod_{\tau=0}^{Mn-1} (t+\tau+1) \cdot  \prod_{1 \leqslant \nu^{\prime} \leqslant p^{l} \atop p \nmid \nu^{\prime},~\nu^{\prime}\neq p^l - \nu}\prod_{\tau=0}^{Mn-1} (p^lt+p^{l}\tau+\nu+\nu^{\prime}) \label{PP}\\
		&\times
		\frac{\prod_{\tau=1}^{w_n}(p^{l}t-p^{l}\tau+\nu)}{\prod_{j=1}^{J}\prod_{\tau=0}^{(M-2\delta_j)n}(p^{l}t+p^{l}(\delta_jn+\tau)+\nu)^{s/J}}. \label{QQ}
	\end{align}
	The first product on line \eqref{PP} is a $\mathbb{Z}_p$-coefficient polynomial in $t$ of degree $Mn$. 
	The second product on line \eqref{PP} and the product on line \eqref{QQ} both belong to $\mathbb{Z}_p\llbracket p^{l}t \rrbracket$. 
	Therefore, we can write $V_n(t+\nu/p^l)$ as
	\begin{align}
		&~ V_n\left(t+\frac{\nu}{p^l}\right) \notag\\
		=&~ p^{(l+1/(p-1))M\varphi(p^{l})n} \cdot \frac{\prod_{j=1}^{J} ((M-2\delta_j)n)!^{s/J}}{n!^{M\varphi(p^{l})}} \cdot \frac{1}{n!^{\lfloor w_n / n\rfloor}(w_n-\lfloor w_n / n\rfloor n)!} \notag \\
		&\times p^{l|\deg V_n|} \cdot p^{lMn} \cdot \sum_{k=0}^{\infty} u_kt^{k},  \label{expression_for_Vn(shift)}
	\end{align}
	where the coefficients $u_k \in \mathbb{Z}_p$ ($k=0,1,2,\ldots$) satisfy
	\begin{equation}\label{v_p(u_k)>}
		v_p(u_k) \geqslant l\max\left\{ 0, k -Mn \right\}.
	\end{equation}
	
	By \eqref{expression_for_Vn(shift)} and Definition \ref{def_L_1}, we have
	\begin{align}
		&~ \mathcal{L}_1\left( V_n\left(t+\frac{\nu}{p^l}\right) \right) \notag\\
		= &~ p^{(l+1/(p-1))M\varphi(p^{l})n} \cdot \frac{\prod_{j=1}^{J} ((M-2\delta_j)n)!^{s/J}}{n!^{M\varphi(p^{l})}} \cdot \frac{1}{n!^{\lfloor w_n / n\rfloor}(w_n-\lfloor w_n / n\rfloor n)!} \notag \\
		&\times p^{l|\deg V_n|} \cdot p^{lMn} \cdot \sum_{k=0}^{\infty} \frac{u_kB_{k+1}}{k+1}. \label{L1Vnshift}
	\end{align}
	By \eqref{v_p(u_k)>} and the von Staudt--Clausen theorem, we have
	\begin{align}
		&~\inf_{k \geqslant 0} v_p\left( \frac{u_kB_{k+1}}{k+1} \right)  \notag\\
		\geqslant&~ \inf_{k \geqslant 0} \left( -1-\frac{\log(k+1)}{\log p} + \max\left\{ 0, k -Mn \right\} \right)  \notag\\
		=& -1-\frac{\log(Mn+1)}{\log p}. \label{inf_vp(L1(u_kt^k))}
	\end{align}
	By \eqref{L1Vnshift} and \eqref{inf_vp(L1(u_kt^k))}, we obtain \eqref{v_p(L_1(V_n(shift)))}. 
	The proof of Lemma \ref{lem_v_p(L_1(V_n(shift)))} is complete.
\end{proof}

\bigskip

Now, we can determine $|T_n|_p$ explicitly for any sufficiently large $n \in \mathbb{Z}_{>0}$.

\begin{lemma}\label{lem_p_ana_Tn}
	We have
	\begin{equation}\label{p_adic_norm_Tn}
		\exp\left( -\alpha_{p,1}(s)n + o(n) \right) \leqslant \left| T_n \right|_p \leqslant \exp\left( -\alpha_{p,2}(s)n + o(n) \right) \quad\text{as~} n \to \infty,
	\end{equation} 
	where $\alpha_{p,1}(s),\alpha_{p,2}(s)$ are given by \eqref{defi_alpha_1} and \eqref{defi_alpha_2}, respectively. Moreover, we have
	\begin{align}
		\alpha_{p,1}(s) &\sim  \frac{s\log s}{J} \sum_{j=1}^{J}(M-2\delta_j) \quad\text{as~} s \to \infty, \label{est_alpha_1}\\
		\alpha_{p,1}(s) - \alpha_{p,2}(s) &= O\left(\frac{s}{\log s}\right) \quad\text{as~} s \to \infty. \label{est_alpha_2}
	\end{align}
\end{lemma}

\begin{proof}
	By Lemmas \ref{lem_p_adic_order_of_sum_W} and \ref{lem_v_p(L_1(V_n(shift)))}, and noting that
	\[ 
	lMn -1 -\frac{\log(Mn+1)}{\log p} > -v_p(|\deg V_n|-1) - 1 
	\]
	for any sufficiently large $n \in \mathbb{Z}_{>0}$, we obtain
	\begin{equation}\label{yeah}
		v_p\left(  \sum_{1 \leqslant \nu \leqslant p^l \atop p \nmid \nu}  \mathcal{L}_1\left( V_n\left(t+\frac{\nu}{p^l}\right) \right)  \right) > v_p\left( \sum_{1 \leqslant \nu \leqslant p^l \atop p \nmid \nu} W_n\left( \frac{\nu}{p^l} \right) \right).
	\end{equation}
	Combining \eqref{yeah} and \eqref{T_n_L1}, we have
	\[ 
	|T_n|_p = \left| \sum_{1 \leqslant \nu \leqslant p^l \atop p \nmid \nu} W_n\left( \frac{\nu}{p^l} \right) \right|_p 
	\]
	for any sufficiently large $n \in \mathbb{Z}_{>0}$. 
	Therefore, \eqref{p_adic_norm_Tn} follows from \eqref{p_adic_norm_of_sumWn}. 
	
	By \eqref{defi_l}, we have
	\begin{align*}
		l\log p &= \log s + O(\log\log s) \quad\text{as~} s\to \infty, \\
		lp^{l} &= O\left( \frac{s}{\log s} \right) \quad\text{as~} s\to \infty.
	\end{align*}
	Therefore, the asymptotic estimates for $\alpha_{p,1}(s)$ and $\alpha_{p,1}(s) -\alpha_{p,2}(s)$ follow.
\end{proof}

\section{Proof of Theorem \ref{mainthm_p}}\label{Sect_proo_p}

We first reduce the proof of Theorem \ref{mainthm_p} to a computational problem.

\begin{proposition}\label{prop_p}
	Let $\psi(\cdot)$ be the digamma function. 
	For any collection of nonnegative integers $(M,\delta_1,\delta_2,\ldots,\delta_J)$ under the constraints $J \geqslant 1$ and
	\[ 
	0 \leqslant \delta_1 \leqslant \delta_2 \leqslant \cdots \leqslant \delta_J < \frac{M}{2}, 
	\]
	we define
	\[ 
	\varpi = \int_{0}^{1} \phi(x)~\mathrm{d}\psi(x) - \int_{0}^{1/(M-2\delta_1)} \phi(x)~\frac{\mathrm{d}x}{x^2},   
	\]
	where the function $\phi(\cdot)$ is given by
	\[ 
	\phi(x) = \inf_{y \in \mathbb{R}} \sum_{j=1}^{J}  \left( \lfloor(M-2\delta_j)x\rfloor - \lfloor y-\delta_j x \rfloor - \lfloor (M-\delta_j)x - y \rfloor \right). 
	\]
	Then, we have
	\[ 
	\dim_{\mathbb{Q}}\operatorname{Span}_{\mathbb{Q}}\left( 1,\zeta_p(3),\zeta_p(5),\ldots,\zeta_p(s-1) \right) \geqslant (C^{\prime}-o(1)) \cdot \log s 
	\]
	as the positive even integer $s \to \infty$, where the constant $C^\prime$ is given by
	\begin{equation}\label{defi_C^prime}
		C^{\prime} = \frac{\sum_{j=1}^{J} (M-2\delta_j) }{\log2 \cdot \sum_{j=1}^{J}(M-2\delta_j)  - \varpi  + J \cdot (M-2\delta_1)}. 
	\end{equation} 
\end{proposition}

\begin{proof}
	For any integer $n>s^2$, consider  
	\[ 
	\widehat{T}_n := \Phi_n^{-s/J} D_{(M-2\delta_1)n}^{s} \cdot T_n.
	\]
	By Lemmas \ref{lem_linear_form_p} and \ref{lem_arith_sigmai}, 
	\[ 
	\widehat{T}_n = \widehat{\sigma}_{n,0} + \sum_{3 \leqslant i \leqslant s-1 \atop i \text{~odd}} \widehat{\sigma}_{n,i}\zeta_p(i) 
	\]
	is a linear form in $1$ and $p$-adic odd zeta values with integer coefficients
	\[ 
	\widehat{\sigma}_{n,i} := \Phi_n^{-s/J} D_{(M-2\delta_1)n}^{s} \cdot \sigma_{n,i} \cdot p^{li} \in \mathbb{Z}, \quad i=0,3,5,\ldots,s-1. 
	\]
	By Lemma \ref{lem_p_ana_Tn} and the facts $v_p(\Phi_n)=0$ for $n>p^2$ and $v_p(D_{(M-2\delta_1)n})=O(\log n)$, we have
	\[ 
	\exp\left( -\widehat{\alpha}_{p,1}(s)n+o(n) \right) \leqslant |\widehat{T}_n|_p \leqslant \exp\left( -\widehat{\alpha}_{p,2}(s)n+o(n) \right) \quad \text{as~} n\to \infty, 
	\]
	where
	\[ 
	\widehat{\alpha}_{p,j}(s) = \alpha_{p,j}(s), \quad j=1,2. 
	\]
	By Lemma \ref{lem_ana_sigma}, Equation \eqref{PNT}, and Lemma \ref{Phi_est}, we have
	\[ 
	\max_{i=0,3,5,\ldots,s-1} |\widehat{\sigma}_{n,i}| \leqslant \exp\left( \widehat{\beta_p}(s)n + o(n) \right) \quad \text{as~} n\to \infty, 
	\]
	where
	\[ 
	\widehat{\beta}_p(s) = \beta_p(s) - \frac{s}{J}\varpi + s(M-2\delta_1). \]
	
	By \eqref{est_alpha_1}, \eqref{est_alpha_2}, and \eqref{beta_p}, we have the following asymptotic estimates for $\widehat{\alpha}_{p,j}(s)$ and $\widehat{\beta}_p(s)$ as $s \to \infty$:
	\begin{align}
		\widehat{\alpha}_{p,1}(s) &\sim  \frac{s\log s}{J}\sum_{j=1}^{J}(M-2\delta_j),  \label{alphahat_1}\\
		\widehat{\alpha}_{p,1}(s) - \widehat{\alpha}_{p,2}(s) &= O\left( \frac{s}{\log s} \right), \label{alphahat_2}\\
		\widehat{\beta}_p(s) &\sim  \log 2 \cdot \frac{s}{J}\sum_{j=1}^{J}(M-2\delta_j) - \frac{s}{J}\varpi + s(M-2\delta_1). \label{betahat_p}
	\end{align}
	Applying Nesterenko's $p$-adic linear independence criterion (Theorem \ref{Nes_p}) to the sequence of linear forms $\{ \widehat{T}_n \}_{n > s^2}$, and using \eqref{alphahat_1}--\eqref{betahat_p}, we obtain
	\begin{align}
		\dim_{\mathbb{Q}}\operatorname{Span}_{\mathbb{Q}}\left( 1,\zeta_p(3),\zeta_p(5),\ldots,\zeta_p(s-1) \right) &\geqslant \frac{\widehat{\alpha}_{p,1}(s)}{\widehat{\beta}_p(s) + \widehat{\alpha}_{p,1}(s) - \widehat{\alpha}_{p,2}(s) } \notag \\
		&= (C^{\prime}-o(1)) \cdot \log s \quad \text{as~} s \to \infty, \label{dim>C^primelogs}
	\end{align}
	where the constant $C^\prime$ is given by
	\begin{equation*}
		C^{\prime} = \frac{\sum_{j=1}^{J} (M-2\delta_j) }{\log2 \cdot \sum_{j=1}^{J}(M-2\delta_j)  - \varpi  + J\cdot(M-2\delta_1)}. 
	\end{equation*}
	
	We have restricted $s$ to $2J\mathbb{Z}_{>0}$ in our constructions of rational functions and linear forms. 
	Thus, \eqref{dim>C^primelogs} holds as $s \in 2J\mathbb{Z}_{>0}$ and $s \to \infty$. 
	However, if we replace the condition `$s \in 2J\mathbb{Z}_{>0}$' by `$s \in 2\mathbb{Z}_{>0}$', inequality  \eqref{dim>C^primelogs} remains valid.
\end{proof}

If we take $M=1$, $\delta_1=0$ (with $J=1$) in Proposition \ref{prop_p}, then we find $\phi(x) \equiv 0$, $\varpi = 0$, and $C^\prime = 1/(1+\log 2)$. 
Thus, we rediscover Sprang's theorem \cite{Spr2020}. 
A simple choice of parameters to improve Sprang's theorem is: $M=6$, $\delta_1=0$, $\delta_2=1$ (with $J=2$). 
This collection of parameters leads to the proof of Theorem \ref{mainthm_p}.

\begin{proof}[Proof of Theorem \textup{\ref{mainthm_p}}]
	Take $M=6$, $\delta_1=0$, $\delta_2=1$ (with $J=2$) in Proposition \ref{prop_p}. 
	As calculated in the proof of Theorem \ref{mainthm} (see Section \ref{Sect_proo}), $\varpi = 2.157479\ldots$\,. 
	Now, by \eqref{defi_C^prime}, we have
	\[ 
	C^{\prime} =  \frac{10}{10\log 2 - \varpi + 12} = \frac{1.009388\ldots}{1+\log 2}. 
	\]
	In conclusion, by Proposition \ref{prop_p}, we have
	\[ 
	\dim_{\mathbb{Q}}\operatorname{Span}_{\mathbb{Q}}\left( 1,\zeta_p(3),\zeta_p(5),\ldots,\zeta_p(s-1) \right) \geqslant \frac{1.009}{1+\log 2} \cdot \log s
	\]
	for any sufficiently large even integer $s$. 
	The proof of Theorem \ref{mainthm_p} is complete.
\end{proof}

\section{Computational explorations}\label{Sect_comp}

In this section, we will present computational results and prove Claims \ref{main_claim} and \ref{thm75}.

Recall that Proposition \ref{prop} and Proposition \ref{prop_p} reduce the theoretic problems to finite computations. 
Since the dependence of $\varpi$ (and hence $C$, $C^{\prime}$) on the parameters is discrete, it is unlikely to find an optimal constant $C$ for Proposition \ref{prop} or an optimal constant $C^{\prime}$ for Proposition \ref{prop_p}.

The parameters $M=6$, $\delta_1 = 0$, $\delta_2 = 1$ (with $J=2$) that we used in the proofs of Theorems \ref{mainthm} and \ref{mainthm_p} are very simple. 
Other simple parameter choices improve the constant $1.009/(1+\log 2)$ in Theorems \ref{mainthm} and \ref{mainthm_p}:
\begin{itemize}
	\item $M=19$, $(\delta_1,\delta_2,\delta_3)=(0,1,2)$. 
	Then, $\varpi = 9.023331\ldots$,  
	\[ 
	C = C^{\prime} = \frac{1.036282\ldots}{1+\log 2}. 
	\]
	\item $M=12$, $(\delta_1,\ldots,\delta_4)=(0,0,1,2)$. 
	Then, $\varpi = 9.363813\ldots$, 
	\[ 
	C = C^{\prime} = \frac{1.049651\ldots}{1+\log 2}. 
	\]
	\item $M=16$, $(\delta_1,\ldots,\delta_5)=(0,0,1,2,3)$. 
	Then, $\varpi=18.818264\ldots$, 
	\[ C = C^{\prime} = \frac{1.062948\ldots}{1+\log 2}. \]
	\item $M=37$, $(\delta_1,\ldots,\delta_{10})=(2,3,\ldots,11)$. 
	Then, $\varpi = 103.306060\ldots$, 
	\[ 
	C=\frac{1.026022\ldots}{1+\log 2}, \quad C^{\prime}= \frac{1.033853\ldots}{1+\log 2}. 
	\] 
	(The last set of parameters is of historical significance. 
	Zudilin used them in \cite{Zud2001} to prove that at least one of $\zeta(5),\zeta(7),\zeta(9),\zeta(11)$ is irrational.)
\end{itemize} 

\begin{remark}
	If $\delta_1 = 0$, then by \eqref{defi_C} and \eqref{defi_C^prime} we have $C=C^{\prime}$. 
	
	If we replace $(M;\delta_1,\ldots,\delta_J)$ by 
	\[
	(\widetilde{M};\widetilde{\delta}_1,\widetilde{\delta}_2,\ldots,\widetilde{\delta}_J) = (M-2\delta_1;\delta_1-\delta_1,\delta_2-\delta_1,\ldots,\delta_J-\delta_1),
	\] 
	then $\widetilde{\varpi}=\varpi$, and $\widetilde{C}=\widetilde{C^\prime}=C^{\prime} \geqslant C$. 
	We may therefore assume $\delta_1 = 0$ without decreasing $C$ and $C^\prime$.
\end{remark}

After an extensive random search in the range $M \leqslant 600$, $\delta_j \leqslant 200$~$(j=1,2,\ldots,J)$, and $J \leqslant 100$, we have found a better constant $C = C^{\prime}$ to be approximately $1.119/(1+\log 2)$. 
For an explanation of how to calculate $\varpi$ using a \texttt{MATLAB} program, we refer the reader to \cite[\S 5]{LZ2022}. 
Now, we give the proof of Claim \ref{main_claim}.

\begin{proof}[Proof of Claim \textup{\ref{main_claim}}]
	Take the parameters $M=433$, $J=89$, and $\delta_1,\ldots,\delta_{89}$ in Table \ref{table1}.
	\begin{table}[htbp]
		\centering
		\begin{tabular}{|lllllll|}
			\hline
			$\delta_1=0$ & $\delta_2=0$ & $\delta_3=0$ & $\delta_4=0$ & $\delta_5=0$ & $\delta_6=0$ & $\delta_7=0$ \\
			$\delta_8=0$ & $\delta_9=0$ & $\delta_{10}=0$ & $\delta_{11}=0$ & $\delta_{12}=0$ & $\delta_{13}=0$ & $\delta_{14}=0$ \\
			$\delta_{15}=0$ & $\delta_{16}=0$ & $\delta_{17}=0$ & $\delta_{18}=0$ & $\delta_{19}=1$ & $\delta_{20}=1$ & $\delta_{21}=2$ \\
			$\delta_{22}=2$ & $\delta_{23}=3$ & $\delta_{24}=3$ & $\delta_{25}=4$ & $\delta_{26}=5$ & $\delta_{27}=6$ & $\delta_{28}=7$ \\
			$\delta_{29}=8$ & $\delta_{30}=9$ & $\delta_{31}=10$ & $\delta_{32}=11$ & $\delta_{33}=12$ & $\delta_{34}=13$ & $\delta_{35}=14$ \\
			$\delta_{36}=15$ & $\delta_{37}=16$ & $\delta_{38}=17$ & $\delta_{39}=18$ & $\delta_{40}=19$ & $\delta_{41}=20$ & $\delta_{42}=21$ \\
			$\delta_{43}=22$ & $\delta_{44}=23$ & $\delta_{45}=25$ & $\delta_{46}=27$ & $\delta_{47}=29$ & $\delta_{48}=31$ & $\delta_{49}=33$ \\
			$\delta_{50}=35$ & $\delta_{51}=37$ & $\delta_{52}=39$ & $\delta_{53}=41$ & $\delta_{54}=43$ & $\delta_{55}=45$ & $\delta_{56}=47$ \\
			$\delta_{57}=49$ & $\delta_{58}=51$ & $\delta_{59}=53$ & $\delta_{60}=55$ & $\delta_{61}=57$ & $\delta_{62}=59$ & $\delta_{63}=61$ \\
			$\delta_{64}=63$ & $\delta_{65}=65$ & $\delta_{66}=67$ & $\delta_{67}=69$ & $\delta_{68}=71$ & $\delta_{69}=73$ & $\delta_{70}=75$ \\
			$\delta_{71}=77$ & $\delta_{72}=79$ & $\delta_{73}=81$ & $\delta_{74}=83$ & $\delta_{75}=85$ & $\delta_{76}=87$ & $\delta_{77}=89$ \\
			$\delta_{78}=91$ & $\delta_{79}=93$ & $\delta_{80}=95$ & $\delta_{81}=97$ & $\delta_{82}=99$ & $\delta_{83}=101$ & $\delta_{84}=103$ \\
			$\delta_{85}=107$ & $\delta_{86}= 111$ & $\delta_{87}=115$ & $\delta_{88}=119$ & $\delta_{89}=123$ & & \\
			\hline
		\end{tabular}
		\caption{The choice of $(\delta_1,\ldots,\delta_{89})$}
		\label{table1}
	\end{table}
	Then, we have $\varpi = 12557.653439\ldots$ and 
	\[ 
	C=C^{\prime} = \frac{1.119356\ldots}{1+\log 2}. 
	\]
	Claim \ref{main_claim} now follows from Propositions \ref{prop} and \ref{prop_p}.
\end{proof}

As a byproduct of their work \cite{BR2001}, Ball and Rivoal proved that
\[ 
\dim_{\mathbb{Q}}\operatorname{Span}_{\mathbb{Q}}\left( 1,\zeta(3),\zeta(5),\ldots,\zeta(169) \right) \geqslant 3. 
\]
In other words, we have the upper bound $\kappa_3 \leqslant 169$ for the least odd integer $\kappa_3 \geqslant 5$ such that $\dim_{\mathbb{Q}}\operatorname{Span}_{\mathbb{Q}}\left( 1,\zeta(3),\zeta(5),\ldots,\zeta(\kappa_3) \right) = 3$. 
Of course, the standard conjecture is $\kappa_3=5$. 
The upper bound for $\kappa_3$ was improved by Zudilin \cite{Zud2002} to $145$ and then improved further to $139$ by Fischler and Zudilin \cite{FZ2010}. 
Using more elaborate parameters, we prove that $\kappa_3 \leqslant 75$.

\begin{proof}[A sketch of the proof of Claim \textup{\ref{thm75}}]
	We use the following rational functions:
	\begin{align*}
		\widetilde{R}_n(t) = &\frac{\prod_{j=1}^{76} ((M-2\delta_j)n)!}{ n!^{2r}} \cdot(2t+Mn)\\
		&\times \frac{(t-rn)_{rn} (t+Mn+1)_{rn}}{\prod_{j=1}^{76}(t+\delta_jn)_{(M - 2\delta_j)n+1}}, \quad n=1,2,3,\ldots,
	\end{align*}
	where $r = 2444$, $M=444$, and $\delta_1,\ldots,\delta_{76}$ are those in Table  \ref{table2}.
	
	\begin{table}[htbp]
		\centering
		\begin{tabular}{|lllllll|}
			\hline
			$\delta_1=1$    & $\delta_{12}=5$  & $\delta_{23}=16$ & $\delta_{34}=30$ & $\delta_{45}=52$ & $\delta_{56}=74$ & $\delta_{67}=96$  \\
			$\delta_2=1$    & $\delta_{13}=6$  & $\delta_{24}=17$ & $\delta_{35}=32$ & $\delta_{46}=54$ & $\delta_{57}=76$ & $\delta_{68}=98$  \\
			$\delta_3=1$    & $\delta_{14}=7$  & $\delta_{25}=18$ & $\delta_{36}=34$ & $\delta_{47}=56$ & $\delta_{58}=78$ & $\delta_{69}=100$ \\
			$\delta_4=1$    & $\delta_{15}=8$  & $\delta_{26}=19$ & $\delta_{37}=36$ & $\delta_{48}=58$ & $\delta_{59}=80$ & $\delta_{70}=102$ \\
			$\delta_5=1$    & $\delta_{16}=9$  & $\delta_{27}=20$ & $\delta_{38}=38$ & $\delta_{49}=60$ & $\delta_{60}=82$ & $\delta_{71}=104$ \\
			$\delta_6=2$    & $\delta_{17}=10$ & $\delta_{28}=21$ & $\delta_{39}=40$ & $\delta_{50}=62$ & $\delta_{61}=84$ & $\delta_{72}=108$ \\
			$\delta_7=2$    & $\delta_{18}=11$ & $\delta_{29}=22$ & $\delta_{40}=42$ & $\delta_{51}=64$ & $\delta_{62}=86$ & $\delta_{73}=112$ \\
			$\delta_8=3$    & $\delta_{19}=12$ & $\delta_{30}=23$ & $\delta_{41}=44$ & $\delta_{52}=66$ & $\delta_{63}=88$ & $\delta_{74}=116$ \\
			$\delta_9=3$    & $\delta_{20}=13$ & $\delta_{31}=24$ & $\delta_{42}=46$ & $\delta_{53}=68$ & $\delta_{64}=90$ & $\delta_{75}=120$ \\
			$\delta_{10}=4$ & $\delta_{21}=14$ & $\delta_{32}=26$ & $\delta_{43}=48$ & $\delta_{54}=70$ & $\delta_{65}=92$ & $\delta_{76}=124$ \\
			$\delta_{11}=4$ & $\delta_{22}=15$ & $\delta_{33}=28$ & $\delta_{44}=50$ & $\delta_{55}=72$ & $\delta_{66}=94$ &                  \\
			\hline
		\end{tabular}
		\caption{The choice of $(\delta_1,\ldots,\delta_{76})$}
		\label{table2}
	\end{table}
	
	Then,  
	\[ 
	\widetilde{S}_n := \sum_{\nu=1}^{\infty} \widetilde{R}_n(\nu) = \widetilde{\rho}_{n,0} + \sum_{3 \leqslant i \leqslant 75 \atop i \text{~odd}} \widetilde{\rho}_{n,i}\zeta(i) 
	\]
	is a linear form in $1$, $\zeta(3)$, $\zeta(5)$, $\ldots$, $\zeta(75)$ with rational coefficients. 
	For any $n>76^2$ we have
	\[ 
	\widetilde{\Phi}_n^{-1} \prod_{j=1}^{76} D_{\max\{ (M-2\delta_1)n,~(M-\delta_j)n \}} \cdot \widetilde{\rho}_{n,0} \in \mathbb{Z} 
	\]
	and
	\[ 
	\widetilde{\Phi}_n^{-1}D_{(M-2\delta_1)n}^{76-i} \cdot \widetilde{\rho}_{n,i} \in \mathbb{Z}, \quad i=3,5,\ldots,75, 
	\]
	where
	\[ 
	\widetilde{\Phi}_{n} = \prod_{\sqrt{(2r+M)n} < q \leqslant (M-2\delta_1)n \atop q \text{~prime}} q^{\widetilde{\phi}(n/q)}, 
	\]
	with $\widetilde{\phi}(x) = \min_{0 \leqslant y < 1} \widetilde{\phi}(x,y)$ and
	\begin{align*}
		\widetilde{\phi}(x,y) =& \lfloor  rx+y \rfloor + \lfloor  (r+M)x - y \rfloor - \lfloor  y \rfloor - \lfloor  Mx-y \rfloor -2r\lfloor  x \rfloor \\
		&+ \sum_{j=1}^{76} \left( \lfloor  (M-2\delta_j)x \rfloor - \lfloor  y-\delta_j x \rfloor - \lfloor  (M-\delta_j)x - y \rfloor \right). 
	\end{align*}
	Note that not only `denominator type elementary bricks' but also `numerator type elementary bricks' contribute to $\widetilde{\Phi}_n$.
	
	As $n \to \infty$, we have the asymptotic estimates: 
	\begin{align*}
		\widetilde{S}_n &= \exp\left( -\widetilde{\alpha} n + o(n) \right), \\
		\widetilde{\Phi}_n &= \exp\left( \widetilde{\varpi}n +o(n) \right), 
	\end{align*}
	where
	\begin{align*}
		\widetilde{\varpi} &= \int_{0}^{1} \widetilde{\phi}(x)\,\mathrm{d}\psi(x) - \int_{0}^{1/(M-2\delta_1)} \widetilde{\phi}(x)~\frac{\mathrm{d}x}{x^2} \\
		&= 42945.452053\ldots,
	\end{align*}
	\begin{align*}
		\widetilde{\alpha} = 
		&- (2r+M)\log\left( 2r+M + x_0 \right) - r\log(r+x_0) \\
		&+ (r+M)\log(r+M+x_0) \\
		&- \sum_{j=1}^{76} (M-2\delta_j)\log(M-2\delta_j)  \\
		&- \sum_{j=1}^{76} (r+\delta_j)\log(r+\delta_j+x_0) \\
		&+ \sum_{j=1}^{76} (r+M-\delta_j)\log(r+M-\delta_j+x_0)  \\
		&= 38489.009014\ldots,
	\end{align*}
	and $x_0 = 0.194387\ldots$ is the unique positive real root of the polynomial
	\begin{align*}
		&~(2r+M+X)(r+X)\prod_{j=1}^{76}(r+\delta_j+X) \\
		-&~ X(r+M+X)\prod_{j=1}^{76}(r+M-\delta_j+X).  
	\end{align*}
	
	On the other hand, we have the bound
	\[ 
	\max_{i=0,3,5,\ldots,75} |\widetilde{\rho}_{n,i}| \leqslant \exp\left( \widetilde{\beta}n + o(n) \right) \quad \text{as~} n \to \infty, 
	\]
	where
	\begin{align*}
		\widetilde{\beta} &= (2r+M)\log\frac{2r+M}{2} - M\log\frac{M}{2} + \log 2 \cdot \sum_{j=1}^{76}(M-2\delta_j) \\ 
		&=58209.043057\ldots.
	\end{align*}
	
	Applying Theorem \ref{reNes} to the sequence of linear forms 
	\[ 
	\left\{\widetilde{\Phi}_n^{-1} \prod_{j=1}^{76} D_{\max\{ (M-2\delta_1)n,~(M-\delta_j)n \}} \cdot \widetilde{S}_n \right\}_{n > 76^2} 
	\] 
	with the divisor $d_{n,i} = D_{443n}^3$ of $\widetilde{\Phi}_n^{-1} \prod_{j=1}^{76} D_{\max\{ (M-2\delta_1)n,~(M-\delta_j)n \}} \cdot \widetilde{\rho}_{n,2i+1}$ for $i=1,2,3,\ldots,37$, we obtain the following bound, with $\gamma_1=443\times 3$:
	\begin{align*}
		&\dim_{\mathbb{Q}}\operatorname{Span}_{\mathbb{Q}}\left( 1,\zeta(3),\zeta(5),\ldots,\zeta(75) \right) \\
		\geqslant & 1 + \frac{\widetilde{\alpha} - \sum_{j=1}^{76}\max\{ M-2\delta_1,~M-\delta_j \} + \widetilde{\varpi} + \gamma_1}{\widetilde{\beta} + \sum_{j=1}^{76}\max\{ M-2\delta_1,~M-\delta_j \} - \widetilde{\varpi}} \\
		=& 2.006260\ldots.
	\end{align*}
	This completes the proof of Claim \ref{thm75}.
\end{proof}

\noindent\textbf{Statement on AI use}: This work was completed in 2024; no AI tools were used. During the 2026 revision process, I used ChatGPT (GPT-6 Astra, OpenAI) to check the proof.

\medskip

\noindent\textbf{Acknowledgements.} I would like to thank Tanguy Rivoal, Johannes Sprang, and Wadim Zudilin for several helpful comments on an early version of this paper. 
I am grateful to the referee for their careful reading and encouraging comments.

\vspace*{3mm}
\begin{flushright}
\begin{minipage}{148mm}\sc\footnotesize
School of Mathematical Sciences, Xiamen University, Fujian, China\\
{\it E--mail address}: {\tt lilaimath@gmail.com} \vspace*{3mm}
\end{minipage}
\end{flushright}

\end{document}